\documentclass[en,oneside,bbfont,article]{amsart} 
\usepackage[lmargin = 30mm, rmargin = 30mm, bmargin = 25mm, tmargin=25mm]{geometry}
\usepackage[utf8]{inputenc}

\makeatletter
\@namedef{subjclassname@2020}{%
  \textup{2020} Mathematics Subject Classification}
\makeatother

\usepackage{mathtools}

\DeclarePairedDelimiter\floor{\lfloor}{\rfloor}

\newtheorem{theorem}{Theorem}[section]
\newtheorem{proposition}[theorem]{Proposition}
\newtheorem{corollary}[theorem]{Corollary}
\newtheorem{lemma}[theorem]{Lemma}
\newtheorem{remark}[theorem]{Remark}

\usepackage{mathrsfs}
\usepackage{amssymb}
\usepackage{dsfont}
\usepackage{amsthm}
\usepackage{amsmath}
\usepackage{comment}
\usepackage{caption,subcaption}
\usepackage{scalefnt}
\usepackage{float,graphicx,verbatim}
\usepackage{tikz}
\usetikzlibrary{calc}
\usepackage{pgfplots}
\pgfplotsset{compat=1.16}
\RequirePackage[colorlinks=true,linkcolor=black,
	citecolor=blue,urlcolor=blue]{hyperref}

\usepackage[backend=biber,style=numeric,citestyle=numeric,maxbibnames=99]{biblatex}
\renewbibmacro{in:}{}
\appto{\bibsetup}{\sloppy}

\newcommand\N{\mathbb{N}}

\newcommand\R{\mathbb{R}}

\newcommand\G{\mathcal{G}}
\newcommand\E{\mathds{E}}
\newcommand\p{\mathds{P}}
\newcommand\GF{\mathfrak{F}}
\newcommand\GI{\mathfrak{I}}
\newcommand\U{\mathrm{U}}
\newcommand\1{\mathds{1}}

\newcommand\oh{\mathrm{o}}
\newcommand\eqd{\overset{d}{=}}
\newcommand\cid{\xrightarrow{d}}
\newcommand\cip{\xrightarrow{\p}}
\newcommand\ciL{\xrightarrow{L^1}}
\newcommand\Var{\mbox{Var}}
\newcommand\ov[1]{\overline{#1}}
\newcommand\un[1]{\underline{#1}}
\newcommand\co{\mathsf{c}}

\newcommand\nf[1]{\normalfont{#1}}
\newcommand\wt{\widetilde}

\title{Asymptotic shape of the concave majorant of a L\'evy process}
\date{\today}
\author{David Bang, Jorge Gonz\'alez C\'azares \& Aleksandar Mijatovi\'c}
\address{Department of Statistics, University of Warwick, \and The Alan Turing Institute, UK}

\email{david.bang@warwick.ac.uk}
\email{jorge.gonzalez-cazares@warwick.ac.uk}
\email{a.mijatovic@warwick.ac.uk}

\begin{document}
\begin{abstract}
We establish distributional limit theorems for the shape statistics of a concave majorant (i.e. the fluctuations of its length,
its supremum, the time it is attained and its value at $T$) of any L\'evy process on $[0,T]$ as $T\to\infty$. The scale of the fluctuations of the length and other statistics, as well as their asymptotic dependence, vary significantly with the tail behaviour of the L\'evy measure. 
The key tool in the proofs is the recent
representation of the concave majorant for all L\'evy processes~\cite{CM_Fluctuation_Levy} using a stick-breaking representation.
\end{abstract}

\subjclass[2020]{60F05; 60G51}

\keywords{concave majorant, convex minorant, limit theorem, stick-breaking process, L\'evy process}

\maketitle

\section{Introduction and main results}
\label{sec:intro}
Convex hulls of random walks and and related processes have been of interest for many decades (see e.g.~\cite{MR2898714,MR2978134,MR3564216,MR3896868,randonfurling2020convex,RWpaper} and references therein). The main objective of the present paper is to understand the asymptotic shape of the concave majorant of a L\'evy process as the time horizon tends to infinity (see Figure~\ref{fig:ConvexHull}). 

\begin{figure}[ht]
\begin{tikzpicture}
\begin{axis} 
	[
	ymin=-1.5,
	ymax=1.8,
	xmin=-60,
	xmax=490,
	xlabel={$T$},
	width=12cm,
	height=6.5cm,
	axis on top=true,
	xticklabels=none,
	yticklabels=none,
	axis x line=none, 
	axis y line=none,
	xlabel style={at={(1,.025)},anchor=south east}
	]			
	
	\node[circle, , fill=none, scale=0.1, label=below:{{\color{gray}$t\mapsto X_t$}}] at (170,.25) {}{};
	
	\node[circle, , fill=none, scale=0.1, label=below:{{\color{blue}$t\mapsto C^\frown_T(t)$}}] at (75,1.70) {}{};
	
	\node[circle, , fill=none, scale=0.1, label=below:{{\color{red}$t\mapsto C^\smile_T(t)$}}] at (90,-.45) {}{};
	
	\node[circle, , fill=blue, scale=0.3, label=above right:{{\color{blue}$(\gamma^\frown_T,\ov{C}^\frown_T)$}}] at (193,1.37) {}{};
	
	\node[circle, , fill=red, scale=0.3, label=below right:{{\color{red}$(\gamma^\smile_T,\un{C}^\smile_T)$}}] at (262,-0.889) {}{};
	
	\node[circle, , fill=gray, scale=0.3, label= right:{{\color{gray}$(T,X_T)$}}] at (400,0.143) {}{};
	
	\node[circle, , fill=gray, scale=0.3, label= left:{{\color{gray}$(0,0)$}}] at (0,0) {}{};
	
	\addplot[
		color = gray,
		]
	coordinates {
(1,0.0)(2,-0.0159)(3,0.0898)(4,0.0815)(5,0.285)(6,0.2)(7,0.123)(8,-0.0293)(9,-0.00265)(10,0.0606)(11,0.0916)(12,0.28)(13,0.422)(14,0.276)(15,0.417)(16,0.418)(17,0.238)(18,0.259)(19,0.217)(20,0.132)(21,0.151)(22,0.189)(23,0.45)(24,0.277)(25,0.174)(26,0.269)(27,0.283)(28,0.449)(29,0.286)(30,0.267)(31,0.121)(32,0.188)(33,0.154)(34,0.314)(35,0.312)(36,0.435)(37,0.341)(38,0.533)(39,0.478)(40,0.54)(41,0.611)(42,0.652)(43,0.833)(44,0.897)(45,0.814)(46,0.782)(47,0.902)(48,0.88)(49,0.857)(50,0.829)(51,0.89)(52,0.929)(53,0.739)(54,0.766)(55,0.748)(56,0.637)(57,0.745)(58,0.756)(59,0.568)(60,0.399)(61,0.346)(62,0.15)(63,0.186)(64,0.146)(65,0.0877)(66,0.131)(67,0.157)(68,0.164)(69,0.171)(70,0.151)(71,0.22)(72,0.262)(73,0.166)(74,0.271)(75,0.342)(76,0.469)(77,0.613)
};
    \addplot[
		color = gray,
		]
	coordinates {
(78,0.909)(79,1.11)(80,1.06)(81,1.05)(82,1.08)(83,0.958)(84,1.01)(85,0.932)(86,1.05)(87,1.16)(88,1.27)(89,1.23)(90,1.12)(91,1.05)(92,1.0)(93,0.968)(94,0.878)(95,1.07)(96,1.16)(97,1.23)(98,1.32)(99,1.26)(100,1.17)(101,1.17)(102,1.07)(103,1.12)(104,0.965)(105,1.06)(106,1.19)(107,1.11)(108,0.997)(109,1.21)(110,1.06)(111,0.947)(112,0.792)(113,0.852)(114,0.678)(115,0.592)(116,0.54)(117,0.485)(118,0.591)(119,0.481)(120,0.529)(121,0.618)(122,0.59)(123,0.48)(124,0.462)(125,0.697)(126,0.664)(127,0.656)(128,0.556)(129,0.796)(130,0.766)(131,0.698)(132,0.796)(133,0.79)(134,0.767)(135,0.849)(136,0.907)(137,0.854)(138,0.803)(139,0.742)(140,0.669)(141,0.635)(142,0.737)(143,0.846)(144,0.804)(145,0.77)(146,0.71)(147,0.603)(148,0.546)(149,0.629)(150,0.729)(151,0.723)(152,0.602)(153,0.825)(154,0.779)(155,0.74)(156,0.78)(157,0.694)(158,0.612)(159,0.632)(160,0.578)(161,0.647)(162,0.717)(163,0.649)(164,0.599)(165,0.561)(166,0.794)(167,0.932)(168,0.944)(169,0.9)(170,0.929)(171,0.997)(172,1.03)(173,1.24)(174,1.37)};
    \addplot[
		color = gray,
		]
	coordinates {
(175,1.29)(176,1.12)(177,1.21)(178,1.09)(179,1.05)(180,0.887)(181,1.05)(182,0.992)(183,1.02)(184,1.05)(185,1.11)(186,1.14)(187,1.04)(188,1.09)(189,1.24)(190,1.32)(191,1.27)(192,1.19)(193,1.37)
};
    \addplot[
		color = gray,
		]
	coordinates {
(194,1.06)(195,0.917)(196,0.942)(197,0.858)(198,0.956)(199,0.95)(200,0.977)(201,0.856)(202,0.703)(203,0.57)(204,0.656)(205,0.736)(206,0.587)(207,0.704)(208,0.531)(209,0.473)
};
    \addplot[
		color = gray,
		]
	coordinates {
(210,0.245)(211,0.238)(212,0.116)(213,0.0396)(214,0.109)(215,0.208)(216,0.294)(217,0.128)(218,-0.106)(219,-0.264)(220,-0.3)(221,-0.324)(222,-0.262)(223,-0.26)(224,-0.342)(225,-0.313)(226,-0.409)(227,-0.483)(228,-0.397)(229,-0.44)(230,-0.522)(231,-0.509)(232,-0.517)(233,-0.547)
};
    \addplot[
		color = gray,
		]
	coordinates {
(234,-0.258)(235,-0.308)(236,-0.305)(237,-0.329)(238,-0.359)(239,-0.408)(240,-0.268)(241,-0.423)(242,-0.491)(243,-0.406)(244,-0.346)(245,-0.473)(246,-0.549)(247,-0.576)(248,-0.535)(249,-0.481)(250,-0.446)(251,-0.577)(252,-0.552)(253,-0.492)(254,-0.568)(255,-0.509)(256,-0.655)(257,-0.779)(258,-0.811)(259,-0.822)(260,-0.764)(261,-0.71)(262,-0.889)(263,-0.851)(264,-0.763)(265,-0.796)(266,-0.784)(267,-0.571)(268,-0.661)(269,-0.652)(270,-0.518)(271,-0.435)(272,-0.433)(273,-0.369)(274,-0.165)(275,-0.123)(276,-0.0464)(277,-0.1)(278,0.00201)(279,0.0476)(280,-0.0175)(281,-0.0458)(282,0.0881)(283,0.232)(284,0.0109)(285,-0.0222)(286,-0.0835)(287,-0.0373)(288,-0.0549)(289,-0.0873)(290,-0.159)(291,-0.0652)(292,-0.0394)(293,-0.183)(294,-0.0184)(295,-0.0474)(296,-0.0932)(297,-0.0382)(298,0.0574)(299,-0.0186)(300,0.0919)(301,-0.0302)(302,-0.054)(303,-0.0175)(304,-0.014)(305,0.019)(306,0.207)(307,0.173)(308,0.241)(309,0.18)(310,0.236)(311,0.235)(312,0.321)(313,0.345)(314,0.379)(315,0.391)(316,0.278)(317,0.318)(318,0.432)(319,0.442)(320,0.473)(321,0.376)(322,0.332)(323,0.238)(324,0.284)(325,0.113)(326,0.121)(327,0.185)(328,0.22)(329,0.0886)(330,0.156)(331,0.182)(332,0.0961)(333,0.0704)(334,-0.0415)(335,-0.0266)(336,0.0334)(337,0.131)(338,0.185)(339,0.261)(340,0.292)(341,0.326)(342,0.138)(343,0.0405)(344,0.105)(345,0.0593)(346,0.12)(347,0.322)(348,0.174)(349,0.264)(350,0.208)(351,0.157)(352,-0.0806)(353,-0.0832)(354,-0.053)(355,-0.0729)(356,-0.238)(357,-0.22)(358,-0.0487)(359,-0.182)(360,-0.182)(361,-0.119)(362,-0.156)(363,-0.127)(364,-0.107)(365,-0.148)(366,-0.308)(367,-0.27)(368,-0.166)(369,-0.0305)(370,-0.0877)(371,-0.0824)(372,-0.185)(373,-0.253)(374,-0.326)(375,-0.391)(376,-0.373)(377,-0.275)(378,-0.318)(379,-0.331)(380,-0.37)(381,-0.336)(382,-0.221)(383,-0.25)(384,-0.247)(385,-0.253)(386,-0.132)(387,-0.126)(388,-0.204)(389,-0.182)(390,0.047)(391,0.0242)(392,-0.0796)
};

    \addplot[
		color = gray,
		]
	coordinates {
(393,0.312)(394,0.229)(395,0.197)(396,0.0164)(397,0.0766)(398,0.0459)(399,0.0201)(400,0.143)
};

    \addplot[
		color = blue,
		]
	coordinates {
(1,0.0)(5,0.285)(13,0.422)(44,0.897)(88,1.27)(98,1.32)(174,1.37)(193,1.37)(393,0.312)(400,0.143)
};

    \addplot[
		color = red,
		]
	coordinates {
(1,0.0)(2,-0.0159)(262,-0.889)(380,-0.37)(399,0.0201)(400,0.143)
};

\end{axis}
\end{tikzpicture}
\caption{\label{fig:ConvexHull} 
A sample path of a L\'evy process $X$ on the interval $[0,T]$, 
the graphs of the concave majorant $C_T^\frown$ and the convex minorant $C_T^\smile$ and the time and space position of their respective supremum 
$(\gamma^\frown_T,\ov{C}^\frown_T)$ and infimum $(\gamma^\smile_T,\un{C}^\smile_T)$.}{} 
\end{figure}
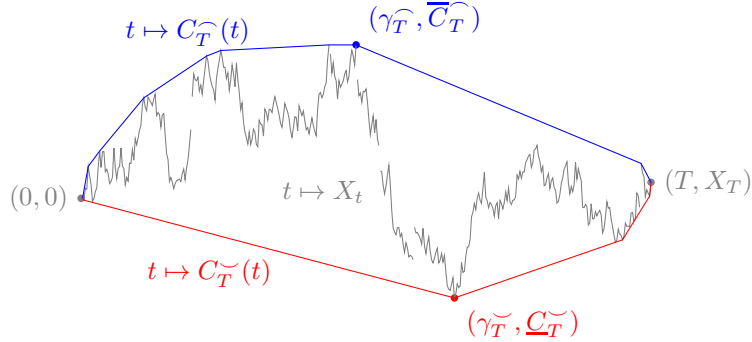

Let $X=(X_t)_{t\ge0}$ be a one-dimensional L\'evy process (see \cite[Def.~1.6, Ch.~1]{SatoBookLevy}) and fix  a time interval $[0,T]$ for some positive time horizon $T>0$. The concave majorant (resp. convex minorant) of a path of a L\'evy process $(X_t)_{t \geq 0}$ is the smallest (resp. largest) function that is point-wise larger (resp. smaller) than the path of $X$, i.e. $C^{\frown}_T(t) \geq X_t$ (resp. $C^{\smile}_T(t) \leq X_t$) for all $t\in [0,T]$. Let $\Upsilon_T^{\frown}$ (resp. $\Upsilon_T^{\smile}$) denote the length 
of the graph of the concave (resp. convex) function $t\mapsto C^{\frown}_T(t)$ (resp. $t\mapsto C^{\smile}_T(t)$) over the interval $[0,T]$. The following inequalities are immediate from Figure~\ref{fig:FacesCM} below:
\begin{equation}
\label{eq:Elementary}
1\le\Upsilon_T^{\frown}/T
\le\Big(T+2
\ov{C}^\frown_T-
C^\frown_T(T)\Big)/T, \qquad\text{where $\ov{C}^\frown_T:=\sup_{t\in[0,T]}C^{\frown}_T(t)$.}
\end{equation} 
If $\E|X_1|^{1+\epsilon}<\infty$ for some $\epsilon>0$ and $\E X_1=0$, 
the bounds in~\eqref{eq:Elementary} and \cite[Prop.~48.10]{SatoBookLevy} imply that $\Upsilon_T^{\frown}/T\to1$ a.s. as $T\to\infty$
(note $\ov{C}^\frown_T=\sup_{t\in [0,T]}X_t$
and 
$C^\frown_T(T)=X_T$).
Our main aim is to identify the precise asymptotic behaviour and the dependence of the shape parameters 
$\Upsilon_T^\frown$, supremum $\ov{C}^\frown_T$,
time of supremum $\gamma^\frown_T$
and final position $C^\frown_T(T)$
of the concave majorant $C^{\frown}_T$. 
More precisely, we seek to identify 
the correct asymptotic mean, analyse the fluctuations of the length $\Upsilon_T^\frown$ around its asymptotic mean and study their dependence on other shape parameters. If the second moment is infinite, we study analogous questions for $X$ in the domain of attraction of a stable process.

Our main result  describes the asymptotic dependence between the fluctuations of the length of the concave majorant, its supremum, final position and the time the supremum is attained, for L\'evy processes that have zero mean and finite variance (see Theorem~\ref{thm:maintheorem1} below). We also describe this dependence in the case the process is in the domain of attraction of a stable law with stability parameter $\alpha\in(0,2]\setminus\{1\}$ (see Theorems~\ref{thm:Theorem2}, \ref{thm:Theorem4} and~\ref{thm:Theorem3} for $\alpha\in(1,2)$ with zero mean, $\alpha\in(1,2]$ with nonzero mean and $\alpha\in(0,1)$, respectively). As we shall see, the dependence has very different structure in each of these cases, with Theorem~\ref{thm:maintheorem1} being the most subtle. In particular, for example, the dependence between the fluctuations of the length of the concave majorant and the other statistics weakens with increasing~$\alpha$. 
For a short overview of the results in this paper see
the YouTube \href{https://youtu.be/b0AOJm-dE3g}{presentation}~\cite{Presentation_DB}.

Before stating our results, recall that the concave majorant of a path of a L\'evy process $X$ is a piecewise linear function with countably many faces (see~\cite[Thm~12]{CM_Fluctuation_Levy}). 
Each face is given by a horizontal length $l>0$ and a vertical height $h\in\R$, thus having the slope $h/l$. 
Note that all the faces with slope equal to a given real value $s\in\R$ must lie next to each other in the graph of the concave majorant and can be concatenated into a \emph{maximal} face with slope $s$. 
Let $H_T$ equal the number of maximal faces with horizontal length $l$ at least $1$. Denote $(x)^+\coloneqq\max(0,x)$ for $x\in\R$ throughout. 

\begin{theorem}
\label{thm:maintheorem1}
Let $X=(X_t)_{t\ge 0}$ be a L\'evy process with L\'evy measure $\nu$. Assume that the L\'evy process has zero mean 
$\E [X_1]=0$ and finite positive variance $\sigma\coloneqq\sqrt{\E[X_1^2]}\in(0,\infty)$. For $T>0$ define $\Theta(T)\coloneqq \frac{1}{2}
    \int_\R x^2\log^+(\min\{T,x^2\})\nu(dx)$. 
Then the following weak limit holds as $T\to\infty$:  
\begin{equation}
\label{eq:CLT-1}
\bigg(
\frac{\Upsilon_T^{\frown}
    -T-(\sigma^2/2)H_T
    +\Theta(T)}{\sqrt{\log T}},
\frac{H_T-\log T}{\sqrt{\log T}},
\frac{\ov{C}^\frown_T}{\sqrt{T}},
\frac{C^\frown_T(T)}{\sqrt{T}},
\frac{\gamma^\frown_T}{T}
\bigg)
\cid 
\Big(
\tfrac{\sigma^2}{\sqrt{2}}Z_1,
Z_2,
\sigma\ov B_1,
\sigma B_1,
\rho
\Big),
\end{equation}
where the standard Brownian motion $B=(B_t)_{t\geq0}$ and the standard normal random variables $Z_1$ and $Z_2$ are independent, $\ov B_1\coloneqq\sup_{t\in[0,1]}B_t$ and $\rho\in[0,1]$ is the a.s. unique time such that $B_\rho=\ov B_1$.
\end{theorem}

The weak limit in~\eqref{eq:CLT-1} shows that the asymptotic centering of the length $\Upsilon_T^\frown$ of the concave majorant $C^\frown_T$ is stochastic. Moreover, the fluctuations around the centering are asymptotically independent of the centering itself and  the randomness in the centering is a function of the horizontal lengths of the faces of $C^\frown_T$ only. A linear transformation of the vector in~\eqref{eq:CLT-1}
yields a deterministic centering of $\Upsilon_T^\frown$
at the cost of increasing the asymptotic variance. Put differently, the variance of the centering contributes 
$\sigma^4/4$ (recall $\sigma^2=\E [X_1^2]$) to the total asymptotic variance of the length $\Upsilon_T^\frown$. For two functions~$f$ and~$g$, write $f(T)=\oh(g(T))$ as $T\to\infty$ if $\lim_{T\to\infty}f(T)/g(T)=0$. 

\begin{corollary}
\label{cor:2nd-moment}
Under the assumptions of Theorem~\ref{thm:maintheorem1}, we have 
\begin{equation}
\label{eq:CLT-2}
\frac{1}{\sqrt{\log T}}
\left(\Upsilon_T^{\frown}
    -T-\frac{\sigma^2}{2}\log T
    +\Theta(T)\right)
\cid \frac{\sqrt{3}}{2}\sigma^2 Z,
\qquad\text{as }T\to\infty,
\end{equation}
where $\Theta(T)=\tfrac{1}{2}\int_\R x^2\log^+(\min\{T,x^2\})\nu(dx)=\oh(\log T)$ and $Z$ is a standard normal variable. Moreover, if $\int_\R x^2\log^+(|x|)^{1/2}\nu(dx)<\infty$, then $\Theta(T)=\oh(\sqrt{\log T})$, and thus 
\begin{equation}
\label{eq:CLT-3}
\frac{1}{\sqrt{\log T}}
\left(\Upsilon_T^{\frown}
    -T-\frac{\sigma^2}{2}\log T\right)
\cid \frac{\sqrt{3}}{2}\sigma^2 Z,
\qquad\text{as }T\to\infty.
\end{equation}
\end{corollary}

Further remarks about Theorem~\ref{thm:maintheorem1} 
and Corollary~\ref{cor:2nd-moment} are in order.

\begin{remark}\normalfont
\textbf{(i)} The limit in~\eqref{eq:CLT-1} reveals that the fluctuations of the asymptotic length of the  concave majorant $C_T^\frown$ are independent of its asymptotic supremum, time of supremum and final position. In the case only the first moment of $X_1$ is finite, the dependence of these shape statistics persists in the limit (see Theorem~\ref{thm:Theorem2} below), while if even the first moment of $X_1$ is infinite, the length $\Upsilon_T^\frown$ becomes a deterministic function of the asymptotic supremum and final position (see Theorem~\ref{thm:Theorem3} below). \\
\noindent \textbf{(ii)} Corollary~\ref{cor:2nd-moment} is stated for the deterministic centering of the length only. However, the
same linear transform yields a quintuple limit analogous to~\eqref{eq:CLT-1}. Put differently, 
as $T\to\infty$, we have
\begin{equation*}
\bigg(
\frac{\Upsilon_T^{\frown}
    -T-(\sigma^2/2)\log T
    +\Theta(T)}{\sqrt{\log T}},
\frac{H_T-\log T}{\sqrt{\log T}},
\frac{\ov{C}^\frown_T}{\sqrt{T}},
\frac{C^\frown_T(T)}{\sqrt{T}},
\frac{\gamma^\frown_T}{T}
\bigg)
\cid 
\Big(
\tfrac{\sigma^2}{\sqrt{2}}Z_1+\tfrac{\sigma^2}{2}Z_2,
Z_2,
\sigma\ov B_1,
\sigma B_1,
\rho
\Big).
\end{equation*}
The dependence structure of the length $\Upsilon_T^\frown$ and $H_T$ is intractable for any finite $T>0$, but, as shown by this limit, is asymptotically rather simple.\\
\noindent \textbf{(iii)}
There exist L\'evy processes for which~\eqref{eq:CLT-2} holds and~\eqref{eq:CLT-3} does not. Indeed, by Fubini's theorem, the integral in the asymptotic mean satisfies 
\[
2\Theta(T)
=\int_\R x^2\log^+(\min\{T, x^2\})\nu(dx)
=\int_1^T\frac{1}{t}\int_{\R\setminus(-\sqrt{t},\sqrt{t})}
    x^2\nu(dx)dt,
\]
and is clearly $\oh(\log T)$ as $T\to\infty$. (Note that $\int_\R x^2\nu(dx)<\infty$ by assumption $\E[X_1^2]<\infty$ and~\cite[Thm~25.3]{SatoBookLevy}.) However, the integral in display can also be arbitrarily close to $\log T$. For instance, setting $\nu(dx):=|x|^{-3}\log(|x|)^{-1}\log\log(|x|)^{-2}dx$, then $\int_{\R\setminus(-t,t)}x^2\nu(dx)=2\log\log(t)^{-1}$ and the integral in display becomes proportional to $(\log\log T)^{-1}\log T$. \\
\noindent\textbf{(iv)} Note that in the weak limit of Theorem~\ref{thm:maintheorem1}  neither $X$ nor its concave majorant $C_T^\frown$ are scaled before the length $\Upsilon_T^\frown$ is calculated. Since $X$ is in the domain of attraction of Brownian motion, one could scale space by $1/\sqrt{T}$ and time by $1/T$ and \emph{then} compute the length of the graph of the resulting concave majorant. This length would, by continuity, converge to the length of the concave majorant of a Brownian motion on $[0,1]$. For the original length $\Upsilon_T^\frown$, this approach only yields  $\Upsilon_T^\frown/T\cid 1$.\\
\noindent\textbf{(v)} To the best of our knowledge, Theorem~\ref{thm:maintheorem1} had been established neither for Brownian motion nor compound Poisson processes. Moreover, the marginal convergence in Corollary~\ref{cor:2nd-moment} does not follow easily from the random walk case, recently analysed in~\cite{RWpaper}, since, for instance, the law of the length of the convex minorant is not invariant under stochastic time-changes, see Figure~\ref{fig:RWvsCP} 
below. 
\end{remark}

A L\'evy process $X$ is in the domain of attraction of an $\alpha$-stable law for some $\alpha\in(0,2]$ if 
\begin{equation}
\label{eq:attract}
X_{T}/a_T\cid S_\alpha(1),
\quad\text{as $T\to\infty$, for some positive function $a_T=T^{1/\alpha}l(T)$,}
\end{equation}
where $l$ is slowly varying (i.e. $l(cx)/l(x)\to 1$ as $x\to\infty$ for all $c>0$) and $(S_\alpha(t))_{t\ge 0}$ is an $\alpha$-stable process (see~\cite[Ch.~3]{SatoBookLevy} for definition).
We note that if $X$ is as in Theorem~\ref{thm:maintheorem1} ($\E[X_1]=0$ and $\sigma=\sqrt{\E[X_1^2]}<\infty$), the standard CLT implies that $X$ is in the domain of attraction of the normal random variable $S_2(1)\sim N(0,\sigma^2)$ with $a_T=\sqrt{T}$. 
Results analogous to Theorem~\ref{thm:maintheorem1} for L\'evy process in the domain of attraction of an $\alpha$-stable law will now be presented: the case $\alpha\in(1,2)$ with $\E[X_1]=0$ (resp. $\alpha\in(1,2]$ with $\E[X_1]\ne 0$; $\alpha\in(0,1)$) is considered in Theorem~\ref{thm:Theorem2} (resp. Theorem~\ref{thm:Theorem4}; Theorem~\ref{thm:Theorem3}). 
To state these theorems, we recall that the uniform stick-breaking process $(\ell_n)_{n\in\N}$ on $[0,1]$ is defined recursively by an iid-$\U(0,1)$ sequence $(V_n)_{n\in\N}$  as follows: $L_0\coloneqq 1$, $\ell_n\coloneqq V_nL_{n-1}$ and $L_n\coloneqq L_{n-1}-\ell_n$ for $n\in\N$. The process $(L_n)_{n\in\N\cup\{0\}}$ will be referred to as the stick-remainders. 

\begin{theorem}\label{thm:Theorem2} 
Assume $X$ is in the domain of attraction of an $\alpha-$stable law with $\alpha\in(1,2)$ and $\E[X_1]=0$. 
Then, as $T\to\infty$, we have 
\begin{equation}\label{eq:Theorem1.2convergence}
\begin{split}
\left(\frac{T}{a_T^2}\right.&\left.
    \left(\Upsilon_T^{\frown}-T \right),
    \frac{\ov{C}^\frown_T}{a_T},
    \frac{C_T^\frown(T)}{a_T},\frac{\gamma_T^\frown}{T} 
    \right)\\  
&\cid\left(\frac{1}{2}\sum_{n=1}^\infty 
    \ell_n^{2/\alpha-1}\big(S_\alpha^{(n)}\big)^2,
    \sum_{n=1}^\infty \ell_n^{1/\alpha} 
    \big( S_\alpha^{(n)}\big)^+,
    \sum_{n=1}^\infty\ell_n^{1/\alpha}S_\alpha^{(n)},
    \sum_{n=1}^\infty\ell_n
    \1_{\{S_\alpha^{(n)}>0\}}
    \right),
\end{split}
\end{equation}
where $(\ell_n)_{n\in\N}$ is a uniform stick-breaking process independent of the sequence $(S_\alpha^{(n)})_{n\in \N}$ of independent copies of $S_\alpha(1)$. 
\end{theorem}

Under the assumptions of Theorem~\ref{thm:Theorem2}, the L\'evy process $X$ has infinite variance.  By~\eqref{eq:Theorem1.2convergence}, the fluctuations of $\Upsilon_T^\frown$ about its centering function are typically of order $T^{2/\alpha-1}$, compared with the fluctuations of order $\sqrt{\log T}$ in the finite variance case (see Theorem~\ref{thm:maintheorem1} above). The last three coordinates of the limit law in~\eqref{eq:Theorem1.2convergence} have the same law as  $(\sup_{t\in[0,1]}S_\alpha(t),S_\alpha(1),\gamma^{\alpha\frown})$, where $\gamma^{\alpha\frown}$ is the time at which the supremum of $S_\alpha(t)$ over $t\in[0,1]$ is attained. We do not know of an interpretation of the law of the first coordinate as a simple functional of the path of the stable process $S_\alpha$. In particular, it is \emph{not} equal to the law of the length of the concave majorant of $S_\alpha$ on $[0,1]$. However, the tail decay of this coordinate can be characterised using the fact that the law of the series $\sum_{n=1}^\infty \ell_n^{2/\alpha-1}(S_\alpha^{(n)})^2$ satisfies a stochastic perpetuity equation.
\begin{proposition}\label{prop:tailbehaviour}
The following asymptotic equivalence holds 
\begin{align*}
\lim_{x\to\infty}\frac{\p\big(\frac{1}{2}\sum_{n=1}^\infty 
    \ell_n^{2/\alpha-1}(S_\alpha^{(n)})^2>x\big)}{\p((S_\alpha^{(1)})^2>x)}
=\lim_{x\to\infty}\frac{\p\big(\frac{1}{2}\sum_{n=1}^\infty 
    \ell_n^{2/\alpha-1}(S_\alpha^{(n)})^2>x\big)}{(c_+ + c_-)x^{-\alpha/2}}
=\frac{2^{1-\alpha/2}}{2-\alpha},
\end{align*} 
for the constants $c_+,c_-\ge 0$ defined by  
$c_\pm\coloneqq \lim_{x\to\infty}\p(\pm S_\alpha^{(1)}>\sqrt{x})/x^{-\alpha/2}$, which satisfy $c_++c_->0$.
\end{proposition}

Note that in Theorems~\ref{thm:maintheorem1} and~\ref{thm:Theorem2}, we have assumed that $X$ has a finite first moment and $\E[X_1]=0$. If the mean is not zero, the behaviour in these cases is described by the following result. In this description, it is important to distinguish between the cases of positive and negative mean. 
\begin{theorem}\label{thm:Theorem4}
Assume $\mu\coloneqq  \E X_1 \ne 0$ and that $X$ is in the domain of attraction of an $\alpha$-stable law with $\alpha \in (1,2]$.

\noindent(a) Suppose $\mu>0$, then we have 
\begin{align}
\bigg( 
\frac{1}{a_T}\left(\Upsilon^\frown_T
    -\sqrt{1+\mu^2}T\right),
\frac{\ov{C}_T^\frown(T)-\mu T}{a_T},
    \frac{C_T^\frown-\mu T}{a_T}\bigg)
\cid
S_\alpha(1)\bigg(
    \frac{\mu}{\sqrt{1+\mu^2}},1,1\bigg),
\qquad\text{as }T \to \infty.
\end{align} 

\noindent(b) Suppose $\mu<0$ and let 
$(\ov{X}_\infty,\gamma_\infty^\frown)$ be the a.s. finite 
limit of the supremum and its time 
$(\ov{C}_T^\frown,\gamma_T^\frown)$ as $T\to\infty$. Then 
\begin{align}
\bigg( 
\frac{1}{a_T}\left(\Upsilon^\frown_T
    -\sqrt{1+\mu^2}T\right),
    \ov{C}_T^\frown,
    \frac{C_T^\frown(T)-\mu T}{a_T}, 
    \gamma_T^\frown\bigg)
\cid
\bigg(
\frac{\mu}{\sqrt{1+\mu^2}}S_\alpha(1),
\ov{X}_\infty,S_\alpha(1),\gamma_\infty^\frown\bigg),
\end{align}  
as $T\to\infty$, where $S_\alpha(1)$ and 
$(\ov{X}_\infty,\gamma_\infty^\frown)$ are independent.
\end{theorem}

Note that the centering function of $\Upsilon_T^\frown$ in Theorem~\ref{thm:Theorem4} equals the length of the graph of the linear function $t\mapsto\mu t$ on $[0,T]$. Moreover, the order of the fluctuations of $\Upsilon_T^\frown$ in this case is different than that in Theorems~\ref{thm:maintheorem1} and~\ref{thm:Theorem2}. Asymptotically, $\Upsilon_T^\frown$ and $C_T^\frown(T)$ are positively correlated when $\mu>0$ and negatively correlated when $\mu<0$. 

When $X$ is in the domain of attraction of an $\alpha$-stable law with $\alpha\in(0,1)$, the tails of $X$ are very heavy. The large jumps of $X$ make its concave majorant thin and tall, implying that the length $\Upsilon_T^\frown$ will be well approximated by the extremes of $X$. 
Define $\underbar{C}_T^\smile\coloneqq \inf_{t \in [0,T]}C_T^\smile(t)$ and let $\gamma_T^\smile$ be the time at which the infimum is attained (see Figure~\ref{fig:ConvexHull}). Denote  $\ov{S}_\alpha(1)\coloneqq \sup_{t\in[0,1]}S_\alpha(t)$, $\un{S}_\alpha(1)\coloneqq \inf_{t\in[0,1]}S_\alpha(t)$ and let 
$\gamma^{\alpha\frown}$ (resp. $\gamma^{\alpha\smile}$) be the time at which $(S_\alpha(t))_{t\in[0,1]}$ attains its supremum (resp. infimum).

\begin{theorem}\label{thm:Theorem3} 
Let $X$ be in the domain of attraction of the $\alpha$-stable law $S_\alpha(1)$ for $\alpha\in(0,1)$. 
Define 
\begin{align*}
&\Lambda_T^1
\coloneqq \left(\frac{\Upsilon_T^{\frown}}{a_T},
    \frac{\ov{C}_T^\frown}{a_T},
    \frac{C_T^\frown(T)}{a_T},
    \frac{\gamma_T^\frown}{T}\right), 
&&\Lambda^1
\coloneqq \left(2\ov{S}_\alpha(1)-S_\alpha(1),
    \ov{S}_\alpha(1),
    S_\alpha(1),
    \gamma^{\alpha\frown}\right),\\
&\Lambda_T^2
\coloneqq \left(\frac{\Upsilon_T^{\smile}}{a_T}, 
    \frac{\underbar{C}_T^\smile}{a_T},
    \frac{C_T^\smile(T)}{a_T},
    \frac{\gamma_T^\smile}{T}\right),
&&\Lambda^2
\coloneqq \left(S_\alpha(1)-2\un{S}_\alpha(t),
    \un{S}_\alpha(t),
    S_\alpha(1),
    \gamma^{\alpha\smile} \right).
\end{align*} 
Then the following joint convergence holds: $\left(\Lambda_T^1,\Lambda_T^2\right)\cid \left(\Lambda^1,\Lambda^2\right)$ as $T\to\infty$.
\end{theorem}

The L\'evy process $X$ in Theorem~\ref{thm:Theorem3} has a thin and tall concave majorant, so the asymptotic centering by $T$, present in Theorems~\ref{thm:maintheorem1} and~\ref{thm:Theorem2}, is no longer required. Moreover, note that in Theorems~\ref{thm:maintheorem1} and~\ref{thm:Theorem2} the fluctuations of $\Upsilon_T^\frown$ about this centering were significantly smaller than $T$, which is no longer the case here. The proof of Theorem~\ref{thm:Theorem3} in Section~\ref{subsec:proof_alpha_less_than_1} below is based on and approximation of $C_T^\frown$ by simpler geometric figures such as the ones in Figure~\ref{fig:FacesCM}.




The concave majorant lies between two natural geometric 
figures. Under the concave majorant lies the `hut' 
$C^\wedge_T$, defined as the linear path connecting the 
vertices: $(0,0)$, $(\gamma_T^\frown,\ov{X}_T)$ and 
$(T,X_T)$, where 
$\gamma_T^\frown=\arg\inf\{t>0:X_t\vee X_{t-}=\ov{X}_T\}$ 
is the time $X$ attains its supremum on $[0,T]$. Over the 
concave majorant lies the `tent' $C^\sqcap_T$, 
defined as the linear path connecting the vertices: 
$(0,0)$, $(0,\ov{X}_T)$, $(T,\ov{X}_T)$ and $(T,X_T)$. 

\begin{figure}[ht]
\begin{tikzpicture}
\begin{axis} 
	[
	ymin=-1,
	ymax=1.5,
	xmin=1,
	xmax=400,
	xlabel={$T$},
	width=10cm,
	height=5.5cm,
	axis on top=true,
	xticklabels=none,
	yticklabels=none,
	axis x line=none, 
	axis y line=none,
	xlabel style={at={(1,.025)},anchor=south east}
	]			
	
	\node[circle, , fill=none, scale=0.1, label=below:{{\color{gray}$X_t$}}] at (200,.25) {}{};
	
	\node[circle, , fill=none, scale=0.1, label=below:{$C^\frown_T$}] at (35,1.25) {}{};
	
	\node[circle, , fill=none, scale=0.1, label=below:{{\color{red}$C^\wedge_T$}}] at (270,.9) {}{};
	
	\node[circle, , fill=none, scale=0.1, label=below:{{\color{blue}$C^\sqcap_T$}}] at (380,1.25) {}{};
	
	\addplot[
		color = gray,
		]
	coordinates {
(1,0.0)(2,-0.0159)(3,0.0898)(4,0.0815)(5,0.285)(6,0.2)(7,0.123)(8,-0.0293)(9,-0.00265)(10,0.0606)(11,0.0916)(12,0.28)(13,0.422)(14,0.276)(15,0.417)(16,0.418)(17,0.238)(18,0.259)(19,0.217)(20,0.132)(21,0.151)(22,0.189)(23,0.45)(24,0.277)(25,0.174)(26,0.269)(27,0.283)(28,0.449)(29,0.286)(30,0.267)(31,0.121)(32,0.188)(33,0.154)(34,0.314)(35,0.312)(36,0.435)(37,0.341)(38,0.533)(39,0.478)(40,0.54)(41,0.611)(42,0.652)(43,0.833)(44,0.897)(45,0.814)(46,0.782)(47,0.902)(48,0.88)(49,0.857)(50,0.829)(51,0.89)(52,0.929)(53,0.739)(54,0.766)(55,0.748)(56,0.637)(57,0.745)(58,0.756)(59,0.568)(60,0.399)(61,0.346)(62,0.15)(63,0.186)(64,0.146)(65,0.0877)(66,0.131)(67,0.157)(68,0.164)(69,0.171)(70,0.151)(71,0.22)(72,0.262)(73,0.166)(74,0.271)(75,0.342)(76,0.469)(77,0.613)
};
    \addplot[
		color = gray,
		]
	coordinates {
(78,0.909)(79,1.11)(80,1.06)(81,1.05)(82,1.08)(83,0.958)(84,1.01)(85,0.932)(86,1.05)(87,1.16)(88,1.27)(89,1.23)(90,1.12)(91,1.05)(92,1.0)(93,0.968)(94,0.878)(95,1.07)(96,1.16)(97,1.23)(98,1.32)(99,1.26)(100,1.17)(101,1.17)(102,1.07)(103,1.12)(104,0.965)(105,1.06)(106,1.19)(107,1.11)(108,0.997)(109,1.21)(110,1.06)(111,0.947)(112,0.792)(113,0.852)(114,0.678)(115,0.592)(116,0.54)(117,0.485)(118,0.591)(119,0.481)(120,0.529)(121,0.618)(122,0.59)(123,0.48)(124,0.462)(125,0.697)(126,0.664)(127,0.656)(128,0.556)(129,0.796)(130,0.766)(131,0.698)(132,0.796)(133,0.79)(134,0.767)(135,0.849)(136,0.907)(137,0.854)(138,0.803)(139,0.742)(140,0.669)(141,0.635)(142,0.737)(143,0.846)(144,0.804)(145,0.77)(146,0.71)(147,0.603)(148,0.546)(149,0.629)(150,0.729)(151,0.723)(152,0.602)(153,0.825)(154,0.779)(155,0.74)(156,0.78)(157,0.694)(158,0.612)(159,0.632)(160,0.578)(161,0.647)(162,0.717)(163,0.649)(164,0.599)(165,0.561)(166,0.794)(167,0.932)(168,0.944)(169,0.9)(170,0.929)(171,0.997)(172,1.03)(173,1.24)(174,1.37)};
    \addplot[
		color = gray,
		]
	coordinates {
(175,1.29)(176,1.12)(177,1.21)(178,1.09)(179,1.05)(180,0.887)(181,1.05)(182,0.992)(183,1.02)(184,1.05)(185,1.11)(186,1.14)(187,1.04)(188,1.09)(189,1.24)(190,1.32)(191,1.27)(192,1.19)(193,1.37)
};
    \addplot[
		color = gray,
		]
	coordinates {
(194,1.06)(195,0.917)(196,0.942)(197,0.858)(198,0.956)(199,0.95)(200,0.977)(201,0.856)(202,0.703)(203,0.57)(204,0.656)(205,0.736)(206,0.587)(207,0.704)(208,0.531)(209,0.473)
};
    \addplot[
		color = gray,
		]
	coordinates {
(210,0.245)(211,0.238)(212,0.116)(213,0.0396)(214,0.109)(215,0.208)(216,0.294)(217,0.128)(218,-0.106)(219,-0.264)(220,-0.3)(221,-0.324)(222,-0.262)(223,-0.26)(224,-0.342)(225,-0.313)(226,-0.409)(227,-0.483)(228,-0.397)(229,-0.44)(230,-0.522)(231,-0.509)(232,-0.517)(233,-0.547)
};
    \addplot[
		color = gray,
		]
	coordinates {
(234,-0.258)(235,-0.308)(236,-0.305)(237,-0.329)(238,-0.359)(239,-0.408)(240,-0.268)(241,-0.423)(242,-0.491)(243,-0.406)(244,-0.346)(245,-0.473)(246,-0.549)(247,-0.576)(248,-0.535)(249,-0.481)(250,-0.446)(251,-0.577)(252,-0.552)(253,-0.492)(254,-0.568)(255,-0.509)(256,-0.655)(257,-0.779)(258,-0.811)(259,-0.822)(260,-0.764)(261,-0.71)(262,-0.889)(263,-0.851)(264,-0.763)(265,-0.796)(266,-0.784)(267,-0.571)(268,-0.661)(269,-0.652)(270,-0.518)(271,-0.435)(272,-0.433)(273,-0.369)(274,-0.165)(275,-0.123)(276,-0.0464)(277,-0.1)(278,0.00201)(279,0.0476)(280,-0.0175)(281,-0.0458)(282,0.0881)(283,0.232)(284,0.0109)(285,-0.0222)(286,-0.0835)(287,-0.0373)(288,-0.0549)(289,-0.0873)(290,-0.159)(291,-0.0652)(292,-0.0394)(293,-0.183)(294,-0.0184)(295,-0.0474)(296,-0.0932)(297,-0.0382)(298,0.0574)(299,-0.0186)(300,0.0919)(301,-0.0302)(302,-0.054)(303,-0.0175)(304,-0.014)(305,0.019)(306,0.207)(307,0.173)(308,0.241)(309,0.18)(310,0.236)(311,0.235)(312,0.321)(313,0.345)(314,0.379)(315,0.391)(316,0.278)(317,0.318)(318,0.432)(319,0.442)(320,0.473)(321,0.376)(322,0.332)(323,0.238)(324,0.284)(325,0.113)(326,0.121)(327,0.185)(328,0.22)(329,0.0886)(330,0.156)(331,0.182)(332,0.0961)(333,0.0704)(334,-0.0415)(335,-0.0266)(336,0.0334)(337,0.131)(338,0.185)(339,0.261)(340,0.292)(341,0.326)(342,0.138)(343,0.0405)(344,0.105)(345,0.0593)(346,0.12)(347,0.322)(348,0.174)(349,0.264)(350,0.208)(351,0.157)(352,-0.0806)(353,-0.0832)(354,-0.053)(355,-0.0729)(356,-0.238)(357,-0.22)(358,-0.0487)(359,-0.182)(360,-0.182)(361,-0.119)(362,-0.156)(363,-0.127)(364,-0.107)(365,-0.148)(366,-0.308)(367,-0.27)(368,-0.166)(369,-0.0305)(370,-0.0877)(371,-0.0824)(372,-0.185)(373,-0.253)(374,-0.326)(375,-0.391)(376,-0.373)(377,-0.275)(378,-0.318)(379,-0.331)(380,-0.37)(381,-0.336)(382,-0.221)(383,-0.25)(384,-0.247)(385,-0.253)(386,-0.132)(387,-0.126)(388,-0.204)(389,-0.182)(390,0.047)(391,0.0242)(392,-0.0796)
};

    \addplot[
		color = gray,
		]
	coordinates {
(393,0.312)(394,0.229)(395,0.197)(396,0.0164)(397,0.0766)(398,0.0459)(399,0.0201)(400,0.143)
};

    \addplot[
		color = black,
		]
	coordinates {
(1,0.0)(5,0.285)(13,0.422)(44,0.897)(88,1.27)(98,1.32)(174,1.37)(193,1.37)(393,0.312)(400,0.143)
};

    \addplot[
		color = red,
		]
	coordinates {
(1,0)(193,1.37)(400,0.143)
};

    \addplot[
		color = blue,
		]
	coordinates {
(1,0)(1,1.37)(400,1.37)(400,0.143)
};

\end{axis}
\end{tikzpicture}
\caption{\label{fig:FacesCM} 
The figure shows a sample of the path of $X$, 
the concave majorant $C_T^\frown$, 
the hut $C_T^\wedge$ and the tent $C_T^\sqcap$.}{}
\end{figure}
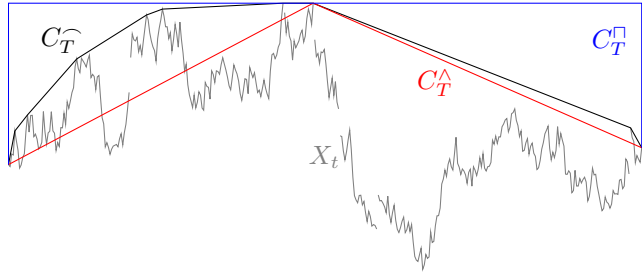

Suppose that the lengths of the hut $C^\wedge_T$ and the tent $C^\sqcap_T$ are $\Upsilon^\wedge_T$ and $\Upsilon^\sqcap_T$, respectively. It is clear from the triangle inequality that $\Upsilon^\wedge_T\le\Upsilon_T^\frown\le\Upsilon^\sqcap_T$. These lengths do not generally all have the same asymptotic behaviour. The next result provides a short comparison in the cases $\alpha\in(1,2]$ with $\E[X_1]=0$ and $\alpha\in(0,1)$.

\begin{proposition}
\label{prop:compare_length}
Define $\Upsilon^\wedge_T$ and  
$\Upsilon^\sqcap_T$ as before then the following statements hold as $T\to\infty$.\\
{\nf{(a)}} Suppose $\E[X_1]=0$ and $\sigma^2=\E[X_1^2]<\infty$, then 
\begin{multline*}
\bigg(\Upsilon^\wedge_T-T,
\frac{1}{\sqrt{\log T}}
\left(\Upsilon_T^{\frown}
    -T-\frac{\sigma^2}{2}\log T
    +\Theta(T)\right),\frac{1}{\sqrt{T}}\big(\Upsilon^\sqcap_T-T\big)\bigg)\\
\cid 
\bigg(
\frac{\sigma^2}{2}\bigg(\frac{\ov B_1^2}{\rho}
    +\frac{(\ov B_1-B_1)^2}{1-\rho}\bigg),
\frac{\sqrt{3}}{2}\sigma^2Z
,\sigma(2\ov B_1-B_1)\bigg),
\end{multline*}
where $Z$ is a standard normal variable independent of the standard Brownian motion $B=(B_t)_{t\ge 0}$, $\ov B_1=\sup_{t\in[0,1]}B_t$ and $\rho\in[0,1]$ is the a.s. unique time such that $B_\rho=\ov B_1$.\\
{\nf{(b)}} Suppose the limit in~\eqref{eq:attract} holds for some $\alpha\in(1,2)$ and $\E[X_1]=0$, then 
\begin{align*}
&\bigg(\frac{T}{a_T^2}\big(\Upsilon^\wedge_T-T\big),
\frac{T}{a_T^2}\big(\Upsilon^\frown_T-T\big)
,\frac{1}{a_T}\big(\Upsilon^\sqcap_T-T\big)\bigg)\\
&\qquad\cid 
\bigg(
\frac{1}{2}\bigg(\sum_{n=1}^\infty 
    \ell_n^{1/\alpha}\big(S_\alpha^{(n)}\big)^+\bigg)^2
+\frac{1}{2}\bigg(\sum_{n=1}^\infty
    \ell_n^{1/\alpha}\big(S_\alpha^{(n)}\big)^-\bigg)^2
,\frac{1}{2}\sum_{n=1}^\infty 
    \ell_n^{2/\alpha-1}\big(S_\alpha^{(n)}\big)^2
,\sum_{n=1}^\infty 
    \ell_n^{1/\alpha}\big|S_\alpha^{(n)}\big|\bigg),
\end{align*}
where $(\ell_n)_{n\in\N}$ is a uniform stick-breaking process independent of the sequence $(S_\alpha^{(n)})_{n\in \N}$ of independent copies of $S_\alpha(1)$.\\
{\nf{(c)}} Suppose the limit in~\eqref{eq:attract} holds for some $\alpha\in(0,1)$, then 
\[
\bigg(\frac{\Upsilon^\wedge_T}{a_T},
    \frac{\Upsilon^\frown_T}{a_T},
    \frac{\Upsilon^\sqcap_T}{a_T}\bigg)
    \cid (2\ov{S}_\alpha(1)-S_\alpha(1))(1,1,1).
\]
\end{proposition}

Under the assumptions of either Theorem~\ref{thm:maintheorem1} or Theorem~\ref{thm:Theorem2}, the centering functions of $\Upsilon_T^\wedge$, $\Upsilon_T^\frown$ and $\Upsilon_T^\sqcap$ 
in Proposition~\ref{prop:compare_length}  
are of the form $T+o(T)$ as $T\to\infty$. However, even though the three statistics are closely related, the order of their fluctuations, measured via the scaling functions, exhibits a wide variety of behaviours, see Table~\ref{tab:compare_length} below. In the case $\alpha\in(0,1)$, the centering functions are all zero and the corresponding scaling functions coincide with the scale of the process. 

\begin{table}[ht]
\begin{tabular}{|c|c|c|c|c|c|} 
	\hline
	Setting & Scaling of $X_T$
	& $\Upsilon_T^\wedge$ & $\Upsilon_T^\frown$ 
	& $\Upsilon_T^\sqcap$\\
	\hline
	Theorem~\ref{thm:maintheorem1}
	$(\E [X_1^2]<\infty)$
	& $a_T=\sqrt{T}$
	& $1$ 
	& $\sqrt{\log T}$
	& $\sqrt{T}$\\
	\hline
	Theorem~\ref{thm:Theorem2} 
    $(1<\alpha<2)$
	& $a_T=T^{1/\alpha}l(T)$
	& $T^{2/\alpha-1}l(T)^2$ 
	& $T^{2/\alpha-1}l(T)^2$
	& $T^{1/\alpha}l(T)$\\
	\hline 
    Theorem~\ref{thm:Theorem3}
    $(0<\alpha<1)$
	& $a_T=T^{1/\alpha}l(T)$
    & $T^{1/\alpha}l(T)$
    & $T^{1/\alpha}l(T)$
    & $T^{1/\alpha}l(T)$\\
	\hline
\end{tabular}
\caption{The table shows the scaling functions (after centering) in the weak limits of the lengths 
$\Upsilon^\wedge_T$, $\Upsilon^\frown_T$ and 
$\Upsilon^\sqcap_T$ under the assumptions of the corresponding theorems with $a_T$ as in~\eqref{eq:attract}.}
\label{tab:compare_length}
\end{table}

Recall that $\Upsilon_T^\wedge\le\Upsilon_T^\frown\le\Upsilon_T^\sqcap$. Interestingly, for $X$ with finite variance, by Proposition~\ref{prop:compare_length}(a) the fluctuations of $\Upsilon_T^\frown$ are asymptotically independent of those of $\Upsilon_T^\wedge$ and $\Upsilon_T^\sqcap$, while the fluctuations of the sandwiching lengths $\Upsilon_T^\wedge$ and $\Upsilon_T^\sqcap$ exhibit a strong asymptotic dependence, both being deterministic functions of the vector $(B_1,\ov B_1,\rho)$.  Proposition~\ref{prop:compare_length}(b)\&(c) states that the dependence of the fluctuations of all three statistics persists in the limit when $\alpha<2$.

\subsection{Overview of the proofs}
\label{sec:roadmap}

Our starting point is~\cite[Thm~12]{CM_Fluctuation_Levy}, which implies the following crucial identity for any L\'evy process and time horizon $T>0$:
\begin{equation}
\label{eq:SB-representation}
(\Upsilon_T^\frown,
H_T',
C_T^\frown(T),
\ov{C}_T^\frown,
\gamma_T^\frown)
\eqd \sum_{n=1}^\infty
\Big(\sqrt{(T\ell_n)^2+\xi_n^2},
\1_{\{T\ell_n\ge 1\}},
\xi_n,
\xi_n^+,
T\ell_n\1_{\{\xi_n>0\}}\Big),
\end{equation} 
where $\xi_n\coloneqq X_{TL_{n-1}}-X_{TL_n}$, $H'_T$ is a random variable such that $|H_T-H_T'|$ is bounded in $L^1$ as $T\to\infty$ (see Lemma~\ref{lem:face_discrepancy} below for details) and $\ell$ is a uniform stick-breaking process independent of $X$ with stick-remainders $(L_n)_{n\in \N \cup \{0\}}$. 
This identity is essential in all that follows as it reduces the claims in Theorems~\ref{thm:maintheorem1}, \ref{thm:Theorem2} and~\ref{thm:Theorem4} to limit statements for the sum in~\eqref{eq:SB-representation}, which is given in terms of the increments of the L\'evy process over independent stick-breaking lengths. Establishing those limits as time horizon $T\to\infty$ turns out to be a delicate task. 

In the case of finite variance and zero mean, the proof of Theorem~\ref{thm:maintheorem1} requires splitting the weak limits into three asymptotically independent weak limits. The faces of $C_T^\frown$ of length smaller than $1$ do not contribute to the fluctuations of  $(\Upsilon_T^\frown,H'_T)$. However, all faces of $C_T^\frown$ of moderate size  contribute in aggregate to its fluctuations, with any finite set of faces the of moderate size not surviving in the limiting fluctuations. In contrast, \emph{only the largest few} faces of $C_T^\frown$ influence the scaling limit of the vector $(C_T^\frown(T),\ov C_T^\frown,\gamma_T^\frown)$, making its limit independent of the limiting fluctuations of $(\Upsilon_T^\frown,H'_T)$. Moreover, the CLT for $(\Upsilon_T^\frown,H'_T)$ consists of two asymptotically independent weak limits. The first captures the fluctuations \emph{due to} the stick-breaking process while the second  describes the fluctuations \emph{conditional on} a manifestation of the stick-breaking process. 
The remaining work in the proof of Theorem~\ref{thm:maintheorem1} is mostly concerned with establishing weak limits, conditional on the stick-breaking process, and crucially depends  on~\cite[Thm~1.1]{Bang_et_al_Gaussian_approximation}.

In the case of finite first moment and infinite variance, the proofs of Theorems~\ref{thm:Theorem2} and~\ref{thm:Theorem4} split the sum in~\eqref{eq:SB-representation} into two sums according to whether the faces are shorter or longer than one. However, unlike in the finite-variance
zero-mean case, 
here this is just a technical step:  
in the proof of Theorems~\ref{thm:Theorem2}
all the faces of the concave majorant survive in the limit, contributing both to the fluctuations of its length as well as  the remaining statistics of $C_T^\frown$. It follows from the 
proof of Theorem~\ref{thm:Theorem4} that
only the vertical heights of the faces of $C_T^\frown$
in aggregate contribute to the fluctuations of its length, which are determined by the asymptotic behaviour of its final point $C_T^\frown(T)=X_T$ as $T\to\infty$.

In the infinite first moment case, Theorem~\ref{thm:Theorem3} follows by a sandwiching argument involving the weak limits for the lengths $\Upsilon_T^\wedge$ and $\Upsilon_T^\sqcap$ as in Proposition~\ref{prop:compare_length} above.
As in the proof of Theorem~\ref{thm:Theorem4}, only the heights of the faces of $C_T^\frown$ in aggregate contribute to this limit. 

\subsection{Connections with the literature}\label{subsec:lit}
Convex hulls of stochastic processes are of longstanding interest, see e.g.~\cite{PitmanBravoCMLP} and the references therein. Of particular interest are the geometric properties of convex hulls such as the length, area and diameter, see~\cite{RWpaper,MR3896868,randonfurling2020convex,MR3272767,MR3385604} for random walks and~\cite{MR2898714} for isotropic stable process. Concave majorants of one-dimensional L\'evy processes are also of interest in physics.  In the monograph~\cite[Ch.~XI]{MR1739699}, for example, the problem of whether a quantum particle stays within the light cone is analysed using concave majorants of one-dimensional L\'evy processes.


If the L\'evy process is in a domain of attraction of a stable law, one can pose two types of question about the limiting behaviour of its convex hulls. A limit of a geometric quantity (e.g. perimeter) of the convex hull of the original process may be considered \textit{or} the limit of the convex hull of the scaled process may be analysed. Since taking a convex hull of a set is a continuous operation, in the latter case it is natural to expect that the limit will be given in terms of the corresponding geometric quantity of the convex hull of the stable limit, which is what happens in~\cite[Sec.~5]{MR3564216}.  The present paper considers the former type of question for the length of the concave majorant. 
It is clear from Theorems~\ref{thm:maintheorem1} and~\ref{thm:Theorem2} above that 
in this case the asymptotic mean and the scale of the fluctuations around them are of different order than those of the process. Moreover, the limit is not given in terms of the corresponding quantity for the stable process. Differently put, we analyse the statistics describing the geometry of the convex hull of the original process as the time horizon tends to infinity
without scaling the process first and then considering the limiting behaviour of such statistics.

The object of study in~\cite{MR3564216} is the convex hull of the scaled multi-dimensional L\'evy process attracted to an isotropic $\alpha$-stable process.
This confers upon the convex hull a spatial homogeneity not enjoyed by the concave majorant, which is a one dimensional object in space-time that behaves very differently in space and time coordinates. A further difference with problem considered in~\cite{MR3564216} is that  our aim is to understand the fluctuations around the asymptotic centering rather than obtaining the limit, which in our case is straightforward, see~\eqref{eq:Elementary} above.


A related question about the fluctuations of the length of the convex minorant of a random walk, as the time horizon tends to infinity, was studied in the recent paper~\cite{RWpaper}. 
CLT-type results for the length of the convex minorant of a random walk ware established in~\cite{RWpaper} under hypothesis analogous to ours (i.e. the increments either have finite variance and zero mean or are in
the domain of attraction of an $\alpha$-stable law for $\alpha \in (0,2)\setminus\{1\}$). The joint limits for the shape statistics in random walk case are not discussed in~\cite{RWpaper}.
Moreover, we stress that the fluctuations of the length of the concave majorant in 
our Theorems~\ref{thm:maintheorem1}, \ref{thm:Theorem2} and~\ref{thm:Theorem4}
cannot be deduced easily from
the results of~\cite{RWpaper} even in the case of a compound Poisson process since the random
time-change connecting it with a random walk  distorts the concave majorant, see Figure~\ref{fig:RWvsCP} below.

\begin{figure}[ht]
\begin{tikzpicture}
\begin{axis} 
	[
	ymin=-2,
	ymax=6,
	xmin=-20,
	xmax=420,
	xlabel={$T$},
	title={Random walk $S_n$},
	width=9cm,
	height=6.5cm,
	axis on top=true,
	xticklabels=none,
	yticklabels=none,
	axis x line=none, 
	axis y line=none,
	xlabel style={at={(1,.025)},anchor=south east}
	]

	\addplot[
		color = gray,
		mark options={scale=.5},
		mark = *,
		only marks,
		]
	coordinates {
(1,0.0)(17,0.498)(33,0.231)(49,-0.503)(65,0.865)(81,0.077)(97,1.28)(113,3.39)(129,4.46)(145,4.5)(161,3.66)(177,1.56)(193,2.35)(209,1.12)(225,0.753)(241,1.8)(257,3.13)(273,4.62)(289,5.38)(305,4.2)(321,3.4)(337,5.52)(353,4.79)(369,3.17)(385,3.85)(401,2.87)
	};
    
    \addplot[
		color = blue,
		]
	coordinates {
(1,0.0)(129,4.46)(289,5.38)(337,5.52)(385,3.85)(401,2.87)
};

\end{axis}
\end{tikzpicture}
\hspace{5mm}
\begin{tikzpicture}
\begin{axis} 
	[
	ymin=-2,
	ymax=6,
	xmin=-20,
	xmax=420,
	xlabel={$T$},
	title={Compound Poisson process $X_t=S_{N_t}$},
	width=9cm,
	height=6.5cm,
	axis on top=true,
	xticklabels=none,
	yticklabels=none,
	axis x line=none, 
	axis y line=none,
	xlabel style={at={(1,.025)},anchor=south east}
	]			
	
	
	\addplot[
		color = gray,
		]
	coordinates {
(1,0.0)(2,0.498)(3,0.498)(4,0.498)(5,0.498)(6,0.498)(7,0.498)(8,0.498)(9,0.498)(10,0.498)(11,0.498)(12,0.498)(13,0.498)(14,0.498)(15,0.498)(16,0.498)(17,0.498)(18,0.498)(19,0.498)(20,0.498)(21,0.498)(22,0.498)(23,0.498)(24,0.498)(25,0.498)(26,0.498)(27,0.498)(28,0.498)(29,0.498)(30,0.498)(31,0.498)(32,0.498)(33,0.498)(34,0.498)(35,0.498)(36,0.498)(37,0.498)(38,0.498)(39,0.231)(40,0.231)(41,0.231)(42,0.231)(43,0.231)(44,0.231)(45,0.231)(46,0.231)(47,0.231)(48,0.231)(49,0.231)(50,0.231)(51,0.231)(52,0.231)(53,0.231)(54,0.231)(55,0.231)(56,0.231)(57,0.231)(58,0.231)(59,0.231)(60,0.231)(61,0.231)(62,-0.503)(63,-0.503)(64,-0.503)(65,-0.503)(66,-0.503)(67,-0.503)(68,-0.503)(69,-0.503)(70,-0.503)(71,-0.503)(72,-0.503)(73,-0.503)(74,-0.503)(75,-0.503)(76,-0.503)(77,-0.503)(78,0.865)(79,0.865)(80,0.865)(81,0.865)(82,0.865)(83,0.865)(84,0.865)(85,0.865)(86,0.865)(87,0.865)(88,0.865)(89,0.865)(90,0.865)(91,0.865)(92,0.077)(93,0.077)(94,0.077)(95,0.077)(96,0.077)(97,0.077)(98,0.077)(99,0.077)(100,1.28)(101,1.28)(102,1.28)(103,1.28)(104,1.28)(105,3.39)(106,3.39)(107,3.39)(108,3.39)(109,3.39)(110,3.39)(111,3.39)(112,3.39)(113,3.39)(114,3.39)(115,3.39)(116,3.39)(117,3.39)(118,3.39)(119,3.39)(120,3.39)(121,3.39)(122,3.39)(123,3.39)(124,3.39)(125,4.46)(126,4.46)(127,4.46)(128,4.5)(129,3.66)(130,3.66)(131,3.66)(132,3.66)(133,3.66)(134,3.66)(135,3.66)(136,1.56)(137,1.56)(138,1.56)(139,1.56)(140,1.56)(141,1.56)(142,1.56)(143,1.56)(144,1.56)(145,1.56)(146,2.35)(147,1.12)(148,1.12)(149,1.12)(150,1.12)(151,1.12)(152,0.753)(153,0.753)(154,0.753)(155,0.753)(156,0.753)(157,0.753)(158,0.753)(159,0.753)(160,0.753)(161,0.753)(162,0.753)(163,0.753)(164,0.753)(165,0.753)(166,0.753)(167,0.753)(168,0.753)(169,1.8)(170,1.8)(171,1.8)(172,1.8)(173,1.8)(174,1.8)(175,1.8)(176,1.8)(177,1.8)(178,3.13)(179,3.13)(180,3.13)(181,3.13)(182,3.13)(183,3.13)(184,3.13)(185,3.13)(186,3.13)(187,4.62)(188,4.62)(189,4.62)(190,4.62)(191,4.62)(192,4.62)(193,4.62)(194,4.62)(195,4.62)(196,4.62)(197,4.62)(198,4.62)(199,4.62)(200,4.62)(201,4.62)(202,4.62)(203,4.62)(204,4.62)(205,4.62)(206,4.62)(207,4.62)(208,4.62)(209,4.62)(210,4.62)(211,4.62)(212,4.62)(213,4.62)(214,4.62)(215,4.62)(216,4.62)(217,4.62)(218,4.62)(219,4.62)(220,4.62)(221,4.62)(222,5.38)(223,4.2)(224,4.2)(225,4.2)(226,4.2)(227,4.2)(228,3.4)(229,3.4)(230,3.4)(231,3.4)(232,3.4)(233,3.4)(234,3.4)(235,3.4)(236,3.4)(237,3.4)(238,3.4)(239,3.4)(240,3.4)(241,3.4)(242,3.4)(243,3.4)(244,3.4)(245,3.4)(246,3.4)(247,3.4)(248,3.4)(249,3.4)(250,3.4)(251,3.4)(252,3.4)(253,5.52)(254,5.52)(255,5.52)(256,5.52)(257,5.52)(258,5.52)(259,5.52)(260,5.52)(261,5.52)(262,5.52)(263,5.52)(264,5.52)(265,5.52)(266,5.52)(267,5.52)(268,5.52)(269,5.52)(270,5.52)(271,5.52)(272,5.52)(273,5.52)(274,4.79)(275,4.79)(276,4.79)(277,4.79)(278,4.79)(279,4.79)(280,4.79)(281,4.79)(282,4.79)(283,4.79)(284,4.79)(285,4.79)(286,4.79)(287,4.79)(288,4.79)(289,4.79)(290,4.79)(291,4.79)(292,4.79)(293,4.79)(294,4.79)(295,4.79)(296,4.79)(297,4.79)(298,4.79)(299,4.79)(300,4.79)(301,4.79)(302,4.79)(303,4.79)(304,4.79)(305,4.79)(306,4.79)(307,4.79)(308,4.79)(309,4.79)(310,4.79)(311,4.79)(312,4.79)(313,4.79)(314,4.79)(315,4.79)(316,4.79)(317,4.79)(318,4.79)(319,4.79)(320,4.79)(321,4.79)(322,4.79)(323,4.79)(324,4.79)(325,4.79)(326,4.79)(327,4.79)(328,4.79)(329,4.79)(330,4.79)(331,4.79)(332,4.79)(333,4.79)(334,4.79)(335,4.79)(336,4.79)(337,4.79)(338,4.79)(339,4.79)(340,4.79)(341,4.79)(342,4.79)(343,4.79)(344,4.79)(345,4.79)(346,4.79)(347,4.79)(348,4.79)(349,4.79)(350,4.79)(351,4.79)(352,4.79)(353,4.79)(354,4.79)(355,4.79)(356,4.79)(357,3.17)(358,3.17)(359,3.85)(360,3.85)(361,3.85)(362,3.85)(363,3.85)(364,3.85)(365,3.85)(366,2.87)(367,2.87)(368,2.87)(369,2.87)(370,2.87)(371,2.87)(372,2.87)(373,2.87)(374,2.87)(375,2.87)(376,2.87)(377,2.87)(378,2.87)(379,2.87)(380,2.87)(381,2.87)(382,2.87)(383,2.87)(384,2.87)(385,2.87)(386,2.87)(387,2.87)(388,2.87)(389,2.87)(390,2.87)(391,2.87)(392,2.87)(393,2.87)(394,2.87)(395,2.87)(396,2.87)(397,2.87)(398,2.87)(399,2.87)(400,2.87)(401,2.87)
};
    
    \addplot[
		color = blue,
		]
	coordinates {
(1,0.0)(2,0.498)(125,4.46)(128,4.5)(222,5.38)(253,5.52)(273,5.52)(356,4.79)(401,2.87)
};

\end{axis}
\end{tikzpicture}
\caption{\label{fig:RWvsCP} 
The figure shows a sample of the path of a 
random walk $S_n$ (left) and that of the compound Poisson 
process $X_t=S_{N_t}$ (right), where $N_t$ is a Poisson 
process independent of $S_n$. Note that both processes visit 
the same states and in the same order, but the random 
time-change induced by $N_t$ distorts the shape of the 
concave majorant, since the two concave majorants have a 
different number of faces.}{} 
\end{figure}
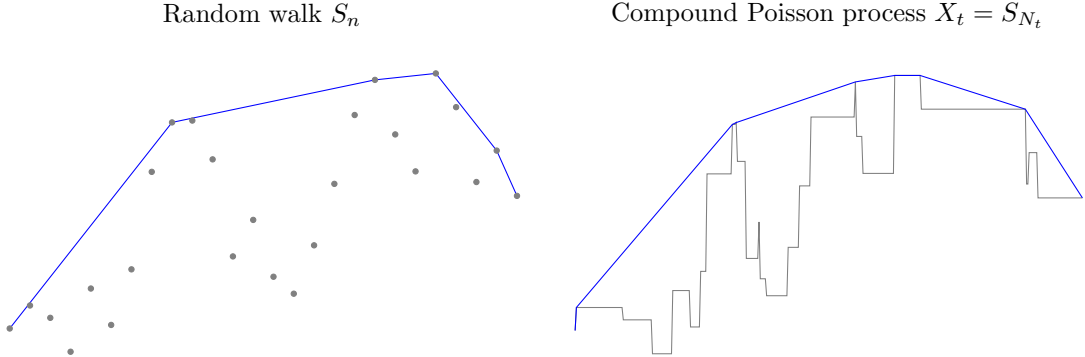

As mentioned in Section~\ref{sec:roadmap} above,
a crucial structure used to  establish our main results is the
characterisation of the law of the concave majorant for all L\'evy processes given in the 
recent article~\cite[Thm~12]{CM_Fluctuation_Levy}. Note that the main result in~\cite{CM_Fluctuation_Levy} generalises to \textit{all} L\'evy processes the characterisation of the law of the concave majorant established in~\cite{MR2978134} 
for diffuse L\'evy processes. This extension is important for the results in the present paper because it allows us to understand the asymptotic shape of the concave majorant of all L\'evy processes, including Poisson processes with drift.

Finally we note that in~\cite[Sec.~28]{SatoBookLevy}, Sato explores the long time behaviour of a L\'evy process and its supremum of the process. Since the concave majorant on $[0,T]$
always coincides with the process at times $T$ and $\gamma_T^\frown$, our results may be viewed as an extension of those in~\cite[Sec.~28]{SatoBookLevy}.




The remainder of the paper is organised as follows. Section~\ref{sec:Proof_Thm1}
proves Theorem~\ref{thm:maintheorem1}. Section~\ref{sec:stable_domain} proves Theorems~\ref{thm:Theorem2}, \ref{thm:Theorem4} and~\ref{thm:Theorem3} as well as the two propositions in the introduction.

\section{Proof of Theorem~\ref{thm:maintheorem1}}
\label{sec:Proof_Thm1}

Recall that $\xi_n= X_{TL_{n-1}}-X_{TL_n}$ and denote $t_n\coloneqq T\ell_n$ for $n\in\N$, where $\ell=(\ell_n)_{n\in\N}$ is a uniform stick-breaking process on $[0,1]$, independent of the L\'evy process $X$, and $L=(L_n)_{n\in\N\cup\{0\}}$ is its stick-remainder process. Note that the sequence $(t_n)_{n\in\N}$ is a uniform stick-breaking process on $[0,T]$. Define the set of indices $\GI_T\coloneqq \{n\in\N:t_n\ge 1\}$. 

The strategy for the proof of Theorem~\ref{thm:maintheorem1} is the following. We will show that the cardinality of $\GI_T$ is by~\cite[Thm~12]{CM_Fluctuation_Levy} closely related to the random variable $H_T$ appearing in Theorem~\ref{thm:maintheorem1} (see Lemma~\ref{lem:face_discrepancy} below for more details). Using this close relationship and~\cite[Thm~12]{CM_Fluctuation_Levy}, we will find that Theorem~\ref{thm:maintheorem1} is equivalent to the vector 
\begin{equation}
\label{eq:quintuple}
\bigg(
\frac{\sum_{n=1}^\infty(\sqrt{\xi_n^2+t_n^2}-t_n)
    -\frac{\sigma^2}{2}|\GI_T|
    +\Theta(T)}{\sqrt{\log T}},
\frac{|\GI_T|-\log T}{\sqrt{\log T}},
\frac{\sum_{n=1}^\infty\xi_n^+}{\sqrt{T}},
\frac{\sum_{n=1}^\infty\xi_n}{\sqrt{T}},
\frac{\sum_{n=1}^\infty t_n\1_{\{\xi_n>0\}}}{T}
\bigg)
\end{equation}
converging weakly to $\zeta\coloneqq(\sigma^2 Z_1/\sqrt{2}, Z_2, \sigma\ov B_1, \sigma B_1, \rho)$ as $T\to\infty$. We next apply 
certain moment estimates for $X$ and limit results for the stick-breaking process~$\ell$ to show that the quintuple in~\eqref{eq:quintuple} converges weakly to $\zeta$ if and only if the following weak limit holds as $T\to\infty$: 
\begin{equation}
\label{eq:CLT-wrt-SB}
\bigg(
\frac{\sum_{n\in\GI_T}(\xi_n^2/t_n-\sigma_{t_n}^2)}
    {2\sqrt{\log T}},
\frac{|\GI_T|-\log T}{\sqrt{\log T}},
\frac{\sum_{n=1}^\infty\xi_n^+}{\sqrt{T}},
\frac{\sum_{n=1}^\infty\xi_n}{\sqrt{T}},
\frac{\sum_{n=1}^\infty t_n\1_{\{\xi_n>0\}}}{T}
\bigg)\cid \zeta,
\end{equation}
where $\sigma_t^2\coloneqq\sigma^2 -\int_{\R\setminus(-\kappa\sqrt{t},\kappa\sqrt{t})} x^2\nu(dx)$ for any $t\ge 1$ and some $\kappa\ge 1$ such that $\sigma_1>0$ (see Proposition~\ref{prop:Upsilon-approx} below for details). Note that the second coordinate in the quintuple in~\eqref{eq:CLT-wrt-SB} is a deterministic function of the stick-breaking process $\ell$ and denote the remaining quadruple by $\zeta'_T$. In order to establish~\eqref{eq:CLT-wrt-SB}, we condition $\zeta'_T$ on $\ell$ and prove that its weak limit under the conditional law is $(\sigma^2Z_1/\sqrt{2}, \sigma\ov B_1, \sigma B_1,\rho)$. Since this limit law does not depend on $\ell$, applying Proposition~\ref{prop:sizeof_n_T} below will complete the proof of Theorem~\ref{thm:maintheorem1}. 

The steps described in this strategy require a variety of technical results. The details of the proof of Theorem~\ref{thm:maintheorem1} are given after the technical results have been established (see page~\pageref{proof:maintheorem1} below). 

 
\subsection{Limit properties of the stick-breaking process}
\label{subsec:SB}
The proof of Theorem~\ref{thm:maintheorem1}
requires a detailed analysis of certain asymptotic properties of the stick-breaking process.
We start with a compensation formula for the point process based on a stick-breaking process, analogous to Campbell's formula for Poisson point processes. 

\begin{lemma}
\label{lem:SB_weak_conv}
Define the point process $\Xi_T\coloneqq \sum_{n\in\N}\delta_{t_n}$, where $\delta_x$ is the Dirac measure at $x$. Then for any measurable function $f:[0,T]\to\R_+$ the following identities hold (with all quantities possibly equal to $+\infty$): 
\begin{equation}
\label{eq:SB_nnt}
\E\bigg[\int_{\R_+} f(x)\Xi_T(dx)\bigg]
=\E\bigg[\sum_{n\in\N}f(t_n)\bigg]
=\int_0^T \frac{f(t)}{t}dt.
\end{equation}
The point process $\Xi_T$ converges weakly as $T\to\infty$ to a Poisson point process $\Xi_\infty$ on $(0,\infty)$ with intensity $t\mapsto t^{-1}$. Moreover, there exists a coupling of point processes $\ov\Xi_\infty$ and $\ov\Xi_T$ for all $T>0$ such that: $\ov\Xi_T\eqd\Xi_T$ and $\ov\Xi_\infty\eqd\Xi_\infty$, $\ov\Xi_T\to\ov\Xi_\infty$ a.s. in the vague topology and for every compact set $A\subset(0,\infty)$, we have $\ov\Xi_T|_A=\ov\Xi_\infty|_A$ for all sufficiently large $T$.
\end{lemma}

The distributional convergence in Lemma~\ref{lem:SB_weak_conv} holds in the vague topology of locally finite measures on $(0,\infty)$, see~\cite[Ch.~16, p.~316]{MR1876169} for definition. More specifically, the a.s. convergence $\ov\Xi_T\to\ov\Xi_\infty$ as $T\to\infty$ in the vague topology is equivalent to $\int f(x)\ov\Xi_T(dx)\to\int f(x)\ov\Xi_\infty(dx)$ for any continuous function $f$ on $(0,\infty)$ that vanishes at $0$ and $\infty$.

\begin{proof}
Note that $-\log\ell_n$ is gamma distributed with density $t\mapsto t^{n-1}e^{-t}/(n-1)!$. Thus, Fubini's theorem implies~\eqref{eq:SB_nnt}: 
\[
\E\left[\sum_{n\in\N}f(t_n)\right]
=\sum_{n\in\N} \int_0^\infty 
    \frac{f(Te^{-t})t^{n-1}}{(n-1)!}e^{-t}dt
=\int_0^\infty f(Te^{-t})dt
=\int_0^T \frac{f(t)}{t}dt.
\]

To prove $\Xi_T\cid\Xi_\infty$ as $T\to\infty$,
 it suffices to provide a coupling 
$(\ov\Xi_T,\ov\Xi_\infty)$ with $\ov\Xi_T\eqd\Xi_T$ and $\ov\Xi_\infty\eqd\Xi_\infty$ such that 
$\ov\Xi_T\to\ov\Xi_\infty$ a.s. as $T\to\infty$.
To that end, 
let $Y$ be a subordinator with infinite mean 
$\E[Y_1]=\infty$ and the convex minorant $C_\infty$ on $\R_+$. By~\cite[Cor.~3]{MR2978134}, for any enumeration of the horizontal lengths $(l_n)_{n\in\N}$ and vertical heights $(h_n)_{n\in\N}$ of the faces of $C_\infty$, the point process $\wt\Xi_\infty\coloneqq\sum_{n\in\N}\delta_{(l_n,h_n)}$ on $(0,\infty)^2$ is Poisson with mean measure $t^{-1}\p(Y_t\in dx)dt$, $(t,x)\in(0,\infty)^2$. Similarly, let 
$\wt\Xi_T$ be the point process of lengths and 
heights of the convex minorant $C_T$ of $Y$ on $[0,T]$. 

For any $s>0$ define the set 
$A_s\coloneqq \{(t,x)\in(0,\infty)^2: x/t<s\}$ and let $T_s$ be 
the last time the right derivative of $C_\infty$ was smaller 
than $s$, which is a.s. finite by~\cite[Cor.~3]{MR2978134}. It follows that $C_T=C_\infty$ on $[0,T_s]$ 
for any $T>T_s$, implying that $\wt\Xi_T$ and 
$\wt\Xi_\infty$ agree on $A_s$ for any $T>T_s$. Since $\bigcup_{s>0}A_s=(0,\infty)^2$ and any compact set in $(0,\infty)^2$ is contained in some $A_s$, we have 
\[
\int_{(0,\infty)^2}f(y)\wt\Xi_T(dy)
=\int_{(0,\infty)^2}f(y)\wt\Xi_\infty(dy),
\qquad T>T_s,
\]
for any compactly supported continuous function $f:(0,\infty)^2\to\R_+$. Since $T_s<\infty$ for all $s>0$, we therefore have $\wt\Xi_T\to\wt\Xi_\infty$ a.s. in the vague topology. Moreover, this implies that the projections
$\ov\Xi_T\coloneqq\wt\Xi_T(\cdot\times\R_+)\eqd\Xi_T$ 
converge to 
$\ov\Xi_\infty\coloneqq\wt\Xi_\infty(\cdot\times\R_+)\eqd\Xi_\infty$ a.s. in the vague topology.
\end{proof}

Recall that $\GI_T=\{n\in\N:t_n\ge 1\}$ is the finite set of indices of sticks in $[0,T]$ of length greater than one and denote by
$\GI_T^\co\coloneqq\N\setminus\GI_T$
its infinite complement.

\begin{corollary}
\label{cor:int_SB_tauT}
{\nf(a)} 
Let $f:\R_+\to\R_+$ be a measurable function
and $T\geq1$. Then the following equalities hold:
\begin{equation}
\label{eq:sum_SB}
\E\sum_{n\in \GI_T}f(t_n)
=\int_1^T \frac{f(t)}{t}dt
\qquad\text{and}\qquad
\E\sum_{n\in \GI_T^\co}f(t_n)
=\int_0^1 \frac{f(t)}{t}dt.
\end{equation}
In particular, the first expectation in~\eqref{eq:sum_SB} always has a (possibly infinite) limit as $T\to\infty$ and for any $q>0$ we have 
$\lim_{T\to\infty}\E\sum_{n\in \GI_T}t_n^{-q}=1/q$. 

{\nf(b)} 
For any bounded and measurable function 
$f:[1,\infty)\to\R$ with $\lim_{t\to\infty}f(t)=0$ 
we have $\E\sum_{n\in \GI_T}f(t_n)=\oh(\log T)$, 
implying that $(\log T)^{-1}\sum_{n\in \GI_T}f(t_n)\ciL 0$.
\end{corollary}

\begin{proof}
(a) Note that $f(t_n)\1_{\{n\in \GI_T\}}=h(t_n)$ where 
$h(t)=f(t)\1_{\{T>1\}}$, so~\eqref{eq:sum_SB} follows 
from~\eqref{eq:SB_nnt}. The formulae for the power functions then follow easily.

(b) Let $T>1$ and note that 
\[
\frac{1}{\log T}\E\sum_{n\in\GI_T}f(t_n)
=\int_1^T \frac{f(t)}{t\log T}dt
=\E[f(Z_T)], 
\]
where $Z_T$ has the density $t\mapsto (t\log T)^{-1}$, $t\in[1,T]$. Since $Z_T\cip\infty$, we have $f(Z_T)\cip 0$ and since the variables $|f(Z_T)|$ are bounded by $\sup_{t\in[1,\infty)}|f(t)|$, the dominated convergence theorem implies that $\E[f(Z_T)]\to 0$, completing the proof. 
\end{proof}

We can now prove the following CLT for the cardinality of the set $\GI_T$ defined above. 

\begin{proposition}
\label{prop:sizeof_n_T}
The cardinality $|\GI_T|$ of the set $\GI_T$ satisfies 
the limits 
\[
|\GI_T|/\log T\ciL 1
\qquad\text{and}\qquad 
(|\GI_T|-\log T)/\sqrt{\log T}\cid N(0,1)
\qquad
\text{as }T\to\infty. 
\]
Moreover, for any $T$ we have  
$\GI_T\subset\{1,\ldots,\tau(T)+1\}$ and 
$\E[\tau(T)]=\E[|\GI_T|]=\log^+(T)$, where we define 
$\tau(T)\coloneqq|\{n\in\N: L_n\ge1/T\}|$.
\end{proposition}

\begin{proof}
Recall by definition of the stick-remainder that  $L_n=\prod_{i=1}^n (1-V_i)$ for an iid sequence $(V_i)_{i\in\N}$ of uniform random variables
on the unit interval. Thus $S_n\coloneqq -\log L_n$ is a random walk with exponential increments of unit mean or, equivalently, the jump times of a Poisson process with unit intensity. Thus, the definition of $\tau(T)$ implies that, for $T>1$, $\tau(T)$ follows the marginal distribution of the Poisson process with unit intensity at time $\log T$. Put differently, $\tau(T)$ is Poisson distributed with mean $\log T$. 
In particular, we have 
$(\tau(T)-\log T)/\sqrt{\log T}\cid N(0,1)$ as $T\to\infty$. 

Recall that $\ell_m=L_n\prod_{i=n+1}^m V_i<L_n$ for all $m>n$. Since $L_{\tau(T)+1}<1/T$ we get $\ell_m<1/T$ for all $m>\tau(T)+1$ and thus $\GI_T\subset\{1,\ldots,\tau(T)+1\}$ and  $\tau(T)+1-|\GI_T|\ge 0$.  Corollary~\ref{cor:int_SB_tauT}(a) gives 
$\E[\tau(T)+1-|\GI_T|]=1$ and thus $\E[|\tau(T)-|\GI_T||]\le 2$ for all $T>0$, implying 
$(\tau(T)-|\GI_T|)/\sqrt{\log T}\ciL 0$. Hence, the CLT for $\tau(T)$ yields the CLT for $|\GI_T|$. Since the random variables $\tau(T)/\log T$, $T\ge 2$, are uniformly integrable, we have $\tau(T)/\log T\ciL 1$ and thus 
\[
|\GI_T|/\log T=(|\GI_T|-\tau(T))/\log T +\tau(T)/\log T \ciL 1.\qedhere
\]
\end{proof}

\begin{remark}\nf
The law $|\GI_T|$
is much more complicated
than that of $\tau(T)$, which follows a Poisson distribution with mean $\log T$ (for $T>1$). The reason for this lies in the fact that $\tau(T)$ is a stopping time in a correct filtration, while $|\GI_T|$ is not, making its moments hard to control. In Proposition~\ref{prop:sizeof_n_T} we circumvent this problem by approximating $|\GI_T|$ with 
$\tau(T)$. We note that, even though the expectation 
$\E\left[\lvert\tau(T)-|\GI_T|\rvert\right]\le 2$
is bounded for all $T>0$, the difference 
$|\tau(T)-|\GI_T||$ 
takes arbitrarily large values with positive probability.
\end{remark}

The following $L^1$ limit holds. 

\begin{proposition}
\label{prop:det_compensator}
Let $f:[1,\infty)\to\R_+$ be measurable and non-increasing with $\lim_{t\to\infty}f(t)=0$. Then 
\[
\frac{1}{\sqrt{\log T}}\bigg(
\sum_{n\in \GI_T}f(t_n)-\E\sum_{n\in\GI_T}f(t_n)\bigg)
=\frac{1}{\sqrt{\log T}}\bigg(
\sum_{n\in \GI_T}f(t_n)-\int_1^T \frac{f(t)}{t}dt\bigg)
\ciL 0,
\quad\text{as }T\to\infty.
\]
\end{proposition}

\begin{proof}
Define for every $T$ the random variables 
\[
A_T\coloneqq \sum_{n\in \GI_T}f(t_n)
    -\sum_{n=1}^{\tau(T)}f(TL_{n}),
\qquad\text{and}\qquad
B_T\coloneqq \sum_{n=1}^{\tau(T)}f(TL_{n})
    -\int_1^T \frac{f(t)}{t}dt,
\]
and note that it suffices to show that $\E|A_T|$ 
is bounded for $T>1$ and $B_T/\sqrt{\log T}\ciL 0$. 

By Lemma~\ref{lem:SB_weak_conv} and the 
equality in law $t_n\eqd TL_n$, we have 
\begin{equation}
\label{eq:E-A_T}
\E[A_T]
=\sum_{n\in\N}\E[f(t_n)\1_{\{t_n\ge 1\}}]
-\sum_{n\in\N}\E[f(TL_n)\1_{\{TL_n\ge 1\}}]
=0.
\end{equation}
Since $f$ is non-increasing
and $t_n\leq T L_{n-1}$ for all $n\in\N$,
we have $C_T\coloneqq \sum_{n\in \GI_T}(f(t_n)-f(TL_{n-1}))\le 0$.  
Similarly, since $f$ is non-increasing, 
Proposition~\ref{prop:sizeof_n_T} gives 
\begin{align*}
|C_T-A_T|
&=\bigg|\sum_{n=2}^{\tau(T)+1}f(TL_{n-1})
    -\sum_{n\in\GI_T}f(TL_{n-1})\bigg|\\
&=\bigg|-f(T)+\sum_{n\in\{1,\ldots,\tau(T)+1\}\setminus \GI_T}f(TL_{n-1})\bigg|
\le f(1)|\tau(T)+2-|\GI_T||.
\end{align*}
Thus~\eqref{eq:E-A_T} and Proposition~\ref{prop:sizeof_n_T} yield $\E[|C_T|]=-\E[C_T]=\E[A_T-C_T]\le 2f(1)$, implying that $\E|A_T|$ is bounded by $4f(1)$ for all $T>1$. 

It remains to show that $B_T/\sqrt{\log T}\ciL 0$. Let $S_n\coloneqq-\log L_n$ and note that 
$\Xi_T\coloneqq\sum_{i=1}^{\tau(T)}\delta_{S_i}$ is a random measure with atoms at the jump times on the interval $[0,\log T]$ of a Poisson process with unit intensity. Thus $\Xi_T$ is a Poisson point process on $[0,\log T]$ with the Lebesgue measure as its mean measure. By the reflection and translation invariance of the Lebesgue measure, the mapping theorem for Poisson point processes~\cite[Sec.~2.3]{MR1207584} gives  $\Xi_T\eqd\sum_{i=1}^{\tau(T)}\delta_{\log T-S_i}$, implying 
\[
D_T\coloneqq \sum_{n=1}^{\tau(T)}f(e^{S_n})
\eqd\sum_{n=1}^{\tau(T)}f(e^{\log T-S_n})
=\sum_{n=1}^{\tau(T)}f(TL_n)
=B_T+\int_1^T \frac{f(t)}{t}dt.
\]
Campbell's formula (see~\cite[p.~28]{MR1207584}) yields 
\[
\E[D_T]=\int_0^{\log T}f(e^x)dx
=\int_1^T\frac{f(t)}{t}dt,
\qquad\text{and}\qquad
\Var[D_T]=\int_0^{\log T}f(e^x)^2dx
=\int_1^T\frac{f(t)^2}{t}dt.
\]
Thus, it suffices to show that 
$\E[B_T^2]/\log T=\Var[D_T]/\log T \to 0$ as $T \to \infty$. Consider the distribution functions $g_T(t)=\log t/\log T$ for $t\in[1,T]$ and define $Z_T\coloneqq g_T^{-1}(U)=T^U$ for all $T>1$ and some fixed uniform random variable $U$ on $(0,1)$. Then $Z_T\to\infty$ a.s. and hence $f(Z_T)^2\to 0$ a.s. as $T\to\infty$. 
By the dominated convergence theorem, 
$\Var[D_T]/\log T=\E[f(Z_T)^2]\to 0$ as $T \to \infty$. 
\end{proof}

\subsection{A conditional limit theorem and the proof of Theorem~\ref{thm:maintheorem1}}
\label{subsec:proof_thm1}

Recall from the first paragraph of Section~\ref{sec:Proof_Thm1} that $(t_n)_{n\in\N}$ denotes a uniform stick-breaking process on $[0,T]$, independent of $X$, and that $\GI_T$ denotes the set $\{n\in\N:t_n\ge 1\}$. Each horizontal length~$t_n$ has an associated slope given by $\xi_n/t_n$, where $\xi_n=X_{TL_{n-1}}-X_{TL_n}$ is the corresponding vertical height. Aggregate all the horizontal lengths with a common slope in the sequence $(t_n)_{n\in\N}$ into a maximal horizontal length corresponding to that slope. Consider the set $\GF_T$ of the maximal horizontal lengths with size at least~1. Note that, by~\cite[Thm~12]{CM_Fluctuation_Levy}, $|\GF_T|\eqd H_T$, where $H_T$ is the number of all horizontal lengths greater or equal to $1$ of the maximal faces of the concave majorant $t\mapsto C_T^\frown(t)$. The analysis of the set $\GF_T$ is based on the properties of the $\GI_T$ established in Subsection~\ref{subsec:SB} above. This strategy is feasible because the difference of sets $\GF_T$ and $\{t_n:n\in\GI_T\}$ is bounded in $L^1$ in the following sense.


\begin{lemma}
\label{lem:face_discrepancy}
For any bounded function 
$f:[1,\infty)\to\R$, the following holds 
\[
\sup_{T>0}\E\Bigg|\sum_{t\in\GF_T}f(t)
-\sum_{n\in\GI_T}f(t_n)\Bigg|<\infty.
\]
\end{lemma}

\begin{proof}
Suppose $X$ is not compound Poisson with drift. Then, by Doeblin's diffuseness lemma~\cite[Lem.~15.22]{MR1876169} and~\cite[Thm~12]{CM_Fluctuation_Levy}, no two slopes in the sequence $(\xi_n/t_n)_{n\in\N}$ coincide, implying the identity $\GF_T=\{t_n:n\in\GI_T\}$ a.s. The claim then follows since both random sums are equal a.s.

Suppose $X$ is compound Poisson with drift $\gamma$ (see~\cite[p.~39]{SatoBookLevy} for the definition of the drift of a L\'evy processes of finite variation). 
Consider two horizontal lengths 
$t_n$ and $t_m$ such that the corresponding slopes $\xi_n/t_n$
and
$\xi_m/t_m$ are equal with positive 
probability. 
Since 
the pair $(t_n,t_m)$ has a density $f_{n,m}:(0,T)\times (0,T)\to(0,\infty)$,
the result in~\cite[Prop.~27.6]{SatoBookLevy} 
implies
\[
\p(\xi_n/t_n=\xi_m/t_m)
=\int_{(0,T)^2}\p(X_s/s=X'_u/u)f_{n,m}(s,u) ds du
=\p(\xi_n/t_n=\gamma=\xi_m/t_m),
\]
where $X'\eqd X$ is a L\'evy process independent of $X$. Thus all slopes $\xi_n/t_n$ different from $\gamma$ are also different from each other with probability one and therefore the corresponding faces are already maximal.  Hence the set equality $\{t_n:n\in \GI_T\}\setminus\GF_T =\{t_n:n\in\GI_T,\, \xi_n=\gamma t_n\}$
holds a.s.

To complete the proof, it is sufficient to show that the number of faces with length at least~1 and slope $\xi_n/t_n=\gamma$ is bounded in $L^1$. By Corollary~\ref{cor:int_SB_tauT}(a), we have 
\[
\E|\{n\in\GI_T:\xi_n/t_n=\gamma\}|
=\E\sum_{n\in \GI_T}\p(X_{t_n}=\gamma t_n|t_n)
=\int_1^T \frac{\p(X_t=\gamma t)}{t}dt
\xrightarrow[T\to\infty]{}\int_1^\infty \frac{\p(X_t=\gamma t)}{t}dt,
\]
where the limit is finite 
by~\cite[Lem.~48.3]{SatoBookLevy}.
\end{proof}

\begin{remark}\nf
The proof of Lemma~\ref{lem:face_discrepancy} implies that
the only maximal face of the concave majorant $C_T^\frown$ of a compound Poisson process $X$ with drift $\gamma$ that corresponds to more than one face in the representation in~\cite[Thm~12]{CM_Fluctuation_Levy} is the face whose slope equals $\gamma$.
All the other faces in this representation are finite in number and have slopes different from each other.
\end{remark}

The following result, a conditional CLT given $\ell$, is the final ingredient for the proof of Theorem~\ref{thm:maintheorem1}.

\begin{proposition}\label{prop:Upsilon-approx}
Suppose $\E[X_1]=0$ and $\sigma\coloneqq \sqrt{\E[X_1^2]}\in(0,\infty)$. If $\nu\ne0$, choose $\kappa\geq1$ such that $\nu((-\kappa,\kappa))\in(0,\infty]$ and otherwise set $\kappa\coloneqq 1$ and recall that $\sigma_t^2=\sigma^2
    -\int_{\R\setminus(-\kappa\sqrt{t},\kappa\sqrt{t})}
    x^2\nu(dx)$ for $t>0$. Then we have the following limit in probability as $T\to\infty$:
\begin{equation}
\label{eq:Upsilon-approx}
\begin{split}
&\Sigma_T
    -\frac{1}{\sqrt{\log T}}\bigg(
    \sum_{n=1}^\infty \big(\sqrt{\xi_n^2+t_n^2}-t_n\big)
        -\frac{1}{2}\bigg(\sigma^2|\GI_T|
        -\int_\R x^2\log^+(\min\{T, x^2\})\nu(dx)\bigg)\bigg)
    \cip 0,\\
&\quad\text{where}\qquad\Sigma_T\coloneqq 
\frac{1}{2\sqrt{\log T}}\sum_{n\in\GI_T}
    \bigg(\frac{\xi_n^2}{t_n}-\sigma_{t_n}^2\bigg).
\end{split}
\end{equation} 
\end{proposition}

\begin{proof} 
Define for every $T>1$, the random variables
\begin{gather*}
\Sigma_T^{(1)}
\coloneqq\frac{1}{\sqrt{\log T}}\sum_{n\in \GI_T^\co}
    \Big(\sqrt{t_n^2 + \xi_n^2} - t_n\Big),
\qquad
\Sigma_T^{(2)}
\coloneqq\frac{1}{\sqrt{\log T}}\sum_{n\in \GI_T}
    \bigg(\sqrt{t_n^2 + \xi_n^2} - t_n-\frac{\xi_n^2}{2t_n}\bigg),\\
\text{and}\qquad
\Sigma_T^{(3)}
\coloneqq\frac{1}{2\sqrt{\log T}}\bigg(\sum_{n\in\GI_T}
    \big(\sigma^2-\sigma_{t_n}^2\big)
        -\int_\R x^2\log^+(\min\{T, x^2\})\nu(dx)\bigg),
\end{gather*}
and note that, since $\N=\GI_T^\co\cup\GI_T$,~\eqref{eq:Upsilon-approx} states that $\Sigma_T^{(3)}-\Sigma_T^{(1)}-\Sigma_T^{(2)}\cip 0$ as $T\to\infty$. It is therefore sufficient to prove that the expectations $\E[|\Sigma_T^{(1)}|]$, $\E[|\Sigma_T^{(2)}|^{1/2}]$ and $\E[|\Sigma_T^{(3)}|]$ all tend to $0$ as $T\to\infty$.

Since $\sqrt{t_n^2+\xi_n^2}-t_n\le|\xi_n|$, Corollary~\ref{cor:int_SB_tauT}(a) implies 
\[
\E[|\Sigma_T^{(1)}|]
\le\frac{1}{\sqrt{\log T}}
    \E\sum_{n\in \GI_T^\co} \E[|\xi_n||\ell]
\le\frac{1}{\sqrt{\log T}}
    \E\sum_{n\in \GI_T^\co} \E[\xi_n^2|\ell]^{1/2}
=\frac{1}{\sqrt{\log T}}\E\sum_{n\in \GI_T^\co} \sigma\sqrt{t_n}\xrightarrow[T\to\infty]{} 0.
\]
Taylor's theorem for the function $x\mapsto\sqrt{1+x^2}$ 
around $x=0$ applied to $\sqrt{1+\xi_n^2/t_n^2}$ yields 
\[
|\Sigma_T^{(2)}|
=\frac{1}{\sqrt{\log T}}\bigg|\sum_{n\in \GI_T}
    \frac{\xi_n^4}{8t_n^3}\cdot
    \theta(|\xi_n|/t_n)\bigg|
\le\frac{1}{\sqrt{\log T}}\sum_{n\in \GI_T}
    \frac{\xi_n^4}{8t_n^3},
\] 
where $\theta:[0,\infty)\to[0,1]$ is a bounded function. 
Recall that $\E[X_t^2]=\Var(X_t)=\sigma^2 t$ for all $t\ge 0$. Since $x\mapsto\sqrt{x}$ is concave and starts at $0$, we have
\begin{align*}
\E\bigg[ \bigg| \frac{1}{\sqrt{\log T}}
    \sum_{n\in \GI_T} t_n^{-3}\xi_n^4\bigg|^{1/2}\bigg]
&\le(\log T)^{-1/4} \E\sum_{n\in \GI_T} 
    t_n^{-3/2}\xi_n^2
=(\log T)^{-1/4}\E\sum_{n\in \GI_T}
    \E\big[t_n^{-3/2}\xi_n^2|\ell\big]\\
&=\sigma^2(\log T)^{-1/4}\E\sum_{n\in \GI_T} 
    t_n^{-1/2}
=2\sigma^2(\log T)^{-1/4}(1-T^{-1/2})
\xrightarrow[T\to\infty]{} 0,
\end{align*} 
where the last equality follows from Corollary~\ref{cor:int_SB_tauT}(a).
This implies that $\E[|\Sigma_T^{(2)}|^{1/2}]\to 0$ as $T\to\infty$.  

It remains to prove that $\E[|\Sigma_T^{(3)}|]\to 0$ as $T\to\infty$. Applying Corollary~\ref{cor:int_SB_tauT}(a) and Fubini's theorem, for any $T>1$ we obtain 
\begin{align*}
\E\sum_{n\in \GI_T}
    (\sigma^2-\sigma^2_{t_n})
&=\int_1^T\frac{1}{t}
    \int_{\R\setminus(-\kappa\sqrt{t},\kappa\sqrt{t})}
        x^2\nu(dx)dt\\
&= \int_{\R\setminus(-\kappa,\kappa)}
    \int_{1}^{T\wedge (x^2/\kappa^2)}
        \frac{dt}{t}x^2\nu(dx)
= \int_\R x^2\log^+(\min\{T, x^2/\kappa^2\})\nu(dx).
\end{align*}
Moreover, since $\kappa\ge 1$, we have 
\begin{align*}
0&\le\int_\R x^2\log^+(\min\{T, x^2\})\nu(dx)
    -\E\sum_{n\in \GI_T}(\sigma^2-\sigma^2_{t_n})\\
&=\int_\R 
    (\log(\kappa^2)
        \1_{\{|x|<\sqrt{T}\}}
    +\log(T\kappa^2/x^2)
        \1_{\{\sqrt{T}\le |x|<\kappa\sqrt{T}\}}
    )x^2\nu(dx)
\le \log(\kappa^2)
    \int_\R x^2\nu(dx)<\infty.
\end{align*}
Thus, Proposition~\ref{prop:det_compensator} implies that  
\[
\begin{split}
\Sigma_T^{(3)}
&=\frac{1}{2\sqrt{\log T}}\bigg(
    \sum_{n\in \GI_T}(\sigma^2-\sigma^2_{t_n})
    -\E \sum_{n\in \GI_T}(\sigma^2-\sigma^2_{t_n})\bigg)\\
&\qquad
    + \frac{1}{2\sqrt{\log T}}\bigg(\E \sum_{n\in \GI_T}(\sigma^2-\sigma^2_{t_n})
    - \int_\R x^2\log^+(\min\{T, x^2\})\nu(dx)\bigg)
\ciL 0,
\quad\text{as }.\qedhere
\end{split}
\]
\end{proof}

The conditional limit result is a key ingredient for the proof of Theorem~\ref{thm:maintheorem1} is the following conditional limit result.

\begin{proposition}
\label{prop:conditional_CLT}
Let $\Sigma_T$ be as in~\eqref{eq:Upsilon-approx} in Proposition~\ref{prop:Upsilon-approx}. Then the following conditional limit holds: for any $x\in\R$,
\begin{equation}
\label{eq:mainCLT2}
\p(\Sigma_T\le x|\ell)
\ciL \Phi(\sqrt{2}x/\sigma^2),
\qquad\text{as }T\to\infty,
\end{equation}
where $\Phi$ is the distribution function of a standard normal random variable.
\end{proposition}

The limit law in~\eqref{eq:mainCLT2} is $N(0,\sigma^4/2)$ and the convergence in $L^1$ is equivalent to the convergence in probability since the random variables $\p(\Sigma_T\le x|\ell)$ are bounded. In particular,~\eqref{eq:mainCLT2} implies the weak convergence $\p(\Sigma_T\le x)\to \Phi(\sqrt{2}x/\sigma^2)$ for all $x\in\R$. The proof of Proposition~\ref{prop:Upsilon-approx} requires certain limit results for stick-breaking processes from  Subsection~\ref{subsec:SB} and~\cite[Thm~1.1]{Bang_et_al_Gaussian_approximation}.

\begin{proof}
The proof of Proposition~\ref{prop:conditional_CLT} consists of three steps.

\textbf{Step 1.} 
Let $Z \sim N(0,1)$ be independent of the stick-breaking process $\ell$. Fix $r>0$ and $\gamma\ge 0$, let 
\[
g_T(t)\coloneqq
(\log T)^{-\gamma/2}|\E[|X_t^2/t|^{\gamma}\1{\{X_t^2/t\le r \sqrt{\log T}\}}
- |\sigma_t^2 Z^2|^{\gamma}\1{\{\sigma_t^2 Z^2\le r \sqrt{\log T}\}}]|,\qquad\text{$t>0$,} 
\]
where we recall that  $\sigma_t^2=\sigma^2-\int_{\R\setminus(-\kappa\sqrt{t},\kappa\sqrt{t})}x^2\nu(dx)$. In this step we establish the following limit: 
\begin{equation}\label{eq:dctfordiff}
\sum_{n\in \GI_T} g_T(t_n)\ciL 0,
    \quad\text{as }T\to\infty.
\end{equation} 
The integration-by-parts formula implies that for any non-negative random variable $\zeta$ and constant $a\in(0,\infty)$ we have 
\[
a^{-\gamma}\E [\zeta^\gamma\1 _{\{\zeta\le a\}}]
=\p(\zeta\le a)
    -\gamma \int_0^1 x^{\gamma-1}
        \p (\zeta \le ax)dx.
\]
Applying the identity in the previous display twice yields 
\begin{align}
\label{eq:summablegDCT}
0&\le g_T(t)
\le r^{\gamma}K_T(t) 
\le 2r^{\gamma}K(t),
\quad\text{where}\\
\nonumber K_T(t)
&\coloneqq
|\p(X_t^2/t\le r\sqrt{\log T})
        -\p(\sigma_t^2 Z^2\le r\sqrt{\log T})| \\
\nonumber&\quad
+\gamma\int_0^1x^{\gamma-1}|\p(X_{t}^2/t\le xr\sqrt{\log T})-\p(\sigma_t Z\le xr\sqrt{\log T})|dx
\end{align} 
and $K(t):=\sup_{x\in\R}|\p(X_t/\sqrt{t}\le x)-\p(\sigma_t Z\le x)|$. 
Since the normal distribution has a bounded density, the weak limits $X_t/\sqrt{t}\cid N(0,\sigma^2)$ and $\sigma_tZ\cid N(0,\sigma^2)$ as $t\to\infty$ hold in the Kolmogorov distance by~\cite[1.8.31~\&~1.8.32,~p.~43]{MR1353441}, implying $\lim_{t\to\infty}K(t)=0$. Moreover, by the dominated convergence theorem, we have $\lim_{T\to\infty}K_T(t)=0$ and thus $\lim_{T\to\infty}g_T(t)=0$ for all $t>0$.  



Let $\ov\Xi_T$ and $\ov\Xi_\infty$ be the coupled point processes described in Lemma~\ref{lem:SB_weak_conv} and recall that $\ov\Xi_T\to\ov\Xi_\infty$ in the vague topology and, for any $N>1$, we have $\ov\Xi_\infty([1,N])<\infty$ and $\ov\Xi_T|_{[1,N]}=\ov\Xi_\infty|_{[1,N]}$ for all sufficiently large $T$. By the definition of vague topology, we have $\int_{[1,\infty)} K(x)\ov\Xi_T(dx)\to \int_{[1,\infty)} K(x)\ov\Xi_\infty(dx)$ a.s. Since $g_T(t)\to 0$ as $T\to\infty$ for every atom $t$ of $\ov\Xi_\infty|_{[1,N]}$, we have 
\begin{multline*}
\limsup_{T\to\infty}\int_{[1,\infty)} g_T(x)\ov\Xi_T(dx)
\le\limsup_{T\to\infty}\int_{[1,N]} g_T(x)\ov\Xi_T(dx)
+ \limsup_{T\to\infty}\int_{(N,\infty)} 2r^\gamma K(x)\ov\Xi_T(dx)\\
=\limsup_{T\to\infty}\int_{[1,N]} g_T(x)\ov\Xi_\infty(dx)
+ \int_{(N,\infty)}2r^\gamma K(x)\ov\Xi_\infty(dx)
= 2r^\gamma \int_{(N,\infty)} K(x)\ov\Xi_\infty(dx).
\end{multline*}
By~\cite[Thm~1.1]{Bang_et_al_Gaussian_approximation}  we have 
$\E\int_{[1,\infty)}K(x)\ov\Xi_\infty(dx)
=\int_1^\infty t^{-1}K(t)dt<\infty$. Therefore, the display above and Fatou's lemma imply 
\begin{align*}
\limsup_{T\to\infty}\E\sum_{n\in\GI_T} g_T(t_n)
&\le \E\limsup_{T\to\infty}\int_{[1,\infty)} g_T(x)\ov\Xi_T(dx)\\
&\le 2r^\gamma \E\int_{(N,\infty)} K(x)\ov\Xi_\infty(dx)
=2r^\gamma \int_{N}^\infty \frac{K(x)}{x}dx
\to 0,\qquad\text{as $N\to\infty$,}
\end{align*}
thus proving~\eqref{eq:dctfordiff}. 

\textbf{Step 2.} 
Denote $S_{n,T}\coloneqq\xi_n^2/(2t_n\sqrt{\log T})$ for all $n\in\N$ and $T>1$. Assume that the following limits in probability hold as $T\to\infty$:
\begin{align}
\label{eq:stepthreeeq1}
&\sum_{n\in \GI_T}
    \p_\ell(S_{n,T}\ge\epsilon)
\cip0,
\quad\text{ for every $\epsilon>0$,}\\
\label{eq:stepthreeeq2}
&\sum_{n\in\GI_T}\Var_\ell\left(S_{n,T}\1_{\{S_{n,T}
        \le r\}}\right)
\cip\frac{\sigma^4}{2},
\quad\text{ for some $r>0$,}\\
\label{eq:stepthreeeq3}
&\sum_{n\in \GI_T}
    \left(\E_\ell\left[S_{n,T}
        \1 _{\{S_{n,T}\le r'\}}\right]
    -\frac{\sigma_{t_n}^2}{2\sqrt{\log T}}\right)
\cip0,
\quad\text{ for some $r'>0$,}
\end{align} 
where we denote $\p_\ell(\cdot)=\p(\cdot|\ell)$, $\E_\ell[\cdot]=\E[\cdot|\ell]$ and  $\Var_\ell(\cdot)\coloneqq \Var(\cdot|\ell)$. We now prove that~\eqref{eq:stepthreeeq1}--\eqref{eq:stepthreeeq3} imply the $L^1$ limit in~\eqref{eq:mainCLT2}.

Since the random variables in~\eqref{eq:mainCLT2} are bounded, it suffices to prove the limit in probability. Fix a sequence $(T_k)_{k\in\N}$ such that $T_k\to\infty$. By a diagonal argument, there exists a subsequence, again denoted $(T_k)_{k\in\N}$ for ease of notation, such that the limit in~\eqref{eq:stepthreeeq1} holds for all positive rational $\epsilon$ as $T_{k}\to\infty$ almost surely. Thus, the limit in~\eqref{eq:stepthreeeq1} holds for \emph{all} $\epsilon>0$ as $T_k\to\infty$ a.s. Moreover, we may assume that the limits in~\eqref{eq:stepthreeeq2}--\eqref{eq:stepthreeeq3} hold a.s. as $T_k\to\infty$. Recall that, given the stick-breaking process $\ell$, the variables $\{S_{n,T_k}:n\in\GI_{T_{k}}\}$ are independent, making $(\{S_{n,T_{k}}:n\in\GI_{T_{k}}\})_{k\in\N}$ a triangular array of row-wise independent random variables. Applying the CLT for triangular arrays in~\cite[Thm~18, Chap.~IV, \S 4]{PetrovSumIndep}, we deduce that~\eqref{eq:mainCLT2} holds a.s. as $T_k\to\infty$. 

\textbf{Step 3.}
In this step we prove~\eqref{eq:stepthreeeq1}--\eqref{eq:stepthreeeq3}. Recall that $Z\sim N(0,1)$ is independent of $\ell$.
By~\eqref{eq:dctfordiff} with $\gamma=0$ and  $r=\epsilon$, Markov's inequality and Proposition~\ref{prop:sizeof_n_T}, we have 
\begin{align*}
\lim_{T\to\infty}\E\sum_{n\in \GI_T}
    \p_\ell\left(\frac{\xi_n^2/t_n}{\sqrt{\log T}}>\epsilon\right)
&=\lim_{T\to\infty}\E\sum_{n\in \GI_T}
    \p_\ell\left(\frac{\sigma_{t_n}^2Z^2}{\sqrt{\log T}}>\epsilon\right)\\
&\le\lim_{T\to\infty}\E\sum_{n\in\GI_T}\frac{\sigma^6_{t_n}\E[Z^6]}{(\epsilon\sqrt{\log T})^{3}} 
\le\lim_{T\to\infty}
\frac{15\sigma^6\E|\GI_T|}{\epsilon^3(\log T)^{3/2}}
=0,
\end{align*}
for all $\epsilon>0$, implying~\eqref{eq:stepthreeeq1} (recall that $S_{n,T}=\xi_n^2/(2t_n\sqrt{\log T})$).

To prove the limit in~\eqref{eq:stepthreeeq2}, first note that $|a^2-b^2|\le (a+b)|a-b|$ for $a,b\ge 0$, implying 
\begin{align*}
&\hspace{16mm}
\Big|\E_\ell\Big[\tfrac{1}{2}t_n^{-1}\xi_n^2
    \1_{\{\xi_n^2\le 2t_n\epsilon\sqrt{\log T}\}}\Big]^2
-\E_\ell\Big[\tfrac{1}{2}\sigma_{t_n}^2Z^2
    \1_{\{\sigma_{t_n}^2Z^2\le 2\epsilon\sqrt{\log T}\}}\Big]^2
    \Big|\\
&\hspace{32mm}
\le 2\epsilon\sqrt{\log T}
\Big|\E_\ell\Big[\tfrac{1}{2}t_n^{-1}\xi_n^2
    \1_{\{\xi_n^2\le 2t_n\epsilon\sqrt{\log T}\}}\Big]
-\E_\ell\Big[\tfrac{1}{2}\sigma_{t_n}^2Z^2
    \1_{\{\sigma_{t_n}^2Z^2\le 2\epsilon\sqrt{\log T}\}}\Big]
\Big|.
\end{align*} 
Thus, by applying~\eqref{eq:dctfordiff} with $\gamma=1$ and $\gamma=2$ and $r=2\epsilon$, we find (all limits are taken in $L^1$): 
\begin{align*}
\lim_{T\to\infty} \frac{1}{\log T}
    \sum_{n\in \GI_T}\Var_\ell\Big(\tfrac{1}{2}t_n^{-1}\xi_n^2
    \1_{\{\xi_n^2
        \le 2t_n\epsilon\sqrt{\log T}\}}\Big)
&=\lim_{T\to\infty} \frac{1}{\log T}
    \sum_{n\in \GI_T}\Var_\ell\Big(\tfrac{1}{2}\sigma_{t_n}^2Z^2
    \1_{\{\sigma_{t_n}^2Z^2
        \le 2\epsilon\sqrt{\log T}\}}\Big)\\
=\lim_{T\to\infty }\frac{1}{\log T}
    \sum_{n\in \GI_T}\Var_\ell\Big(\tfrac{1}{2}\sigma_{t_n}^2Z^2\Big)
&=\lim_{T\to\infty}\frac{1}{2\log T}\bigg(
    \sigma^4|\GI_T|
    +\sum_{n\in \GI_T}(\sigma_{t_n}^4-\sigma^4)\bigg)
=\frac{\sigma^4}{2},
\end{align*} 
where the first equality in the second line follows from the fact that $\Var(Z^2)=2$ and the last equality in the same line follows from 
Proposition~\ref{prop:sizeof_n_T} and Corollary~\ref{cor:int_SB_tauT}(b) applied to the bounded function $t\mapsto\sigma_t^4-\sigma^4$ with zero limit as $t\to\infty$. This establishes~\eqref{eq:stepthreeeq2} since  $S_{n,T}=\xi_n^2/(2t_n\sqrt{\log T})$.

It remains to prove~\eqref{eq:stepthreeeq3}. Markov's inequality, the equality $\E[Z^2]=1$ and Proposition~\ref{prop:sizeof_n_T} imply
\begin{align*}
&\frac{1}{\sqrt{\log T}} \sum_{n\in \GI_T}\bigg|
    \E_\ell\Big[\tfrac{1}{2}\sigma_{t_n}^2Z^2
        \1_{\{\sigma_{t_n}^2Z^2\le 2\epsilon\sqrt{\log T}\}}\Big] 
    -\frac{\sigma_{t_n}^2}{2}\bigg|\\
&\qquad\qquad
=\frac{1}{\sqrt{\log T}} 
    \sum_{n\in \GI_T}\E_\ell\Big[ 
    \tfrac{1}{2}\sigma_{t_n}^2Z^2\1_{\{\sigma_{t_n}^2Z^2>2\epsilon \sqrt{\log T}\}}\Big]
\le\frac{1}{\sqrt{\log T}}\sum_{n\in \GI_T}
    \frac{\E_\ell[\sigma_{t_n}^8Z^8]}{(2\epsilon\sqrt{\log T})^3}\\
&\qquad\qquad
=\frac{1}{8\epsilon^3\log^2 T}\sum_{n\in \GI_T}
    \sigma_{t_n}^8\E[Z^8]
\le\frac{105\sigma^8}{8\epsilon^3\log^2 T} |\GI_T|
\ciL 0.
\end{align*}
The display above and~\eqref{eq:dctfordiff} with $\gamma=1$ and $r=2\epsilon$ imply~\eqref{eq:stepthreeeq3}, completing the proof. 
\end{proof}

\begin{proof}[Proof of Theorem~\ref{thm:maintheorem1}]
\label{proof:maintheorem1}
The proof of Theorem~\ref{thm:maintheorem1} consists of several steps.

\textbf{Step 1.}
In this step we show that~\eqref{eq:CLT-1} follows from the limits in~\eqref{eq:quintuple-Billingsley} below. By~\cite[Thm~12]{CM_Fluctuation_Levy}, Lemma~\ref{lem:face_discrepancy} and Proposition~\ref{prop:Upsilon-approx}, the weak limit in~\eqref{eq:CLT-1} of Theorem~\ref{thm:maintheorem1} is equivalent to the following limit as $T\to\infty$:
\begin{equation}
\label{eq:quintuple-conv}
\zeta_T
\coloneqq\left(\frac{\sum_{n\in\GI_T} (\xi_n^2/t_n-\sigma_{t_n}^2)}{2\sqrt{\log T}},
\frac{|\GI_T|-\log T}{\sqrt{\log T}},
\frac{\sum_{n=1}^\infty \xi_n^+}{\sqrt{T}},
\frac{\sum_{n=1}^\infty \xi_n}{\sqrt{T}},
\frac{\sum_{n=1}^\infty t_n\1_{\{\xi_n>0\}}}{T}
\right)
\cid \zeta,
\end{equation}
where $\zeta=(\sigma^2Z_1/\sqrt{2},Z_2,\sigma\ov{B}_1, \sigma B_1,\rho)$, the standard Brownian motion $B$, the stick-breaking process $\ell$ and the standard normal variables $Z_1$ and $Z_2$ are all independent. 

Define $\eta_n\coloneqq \xi_n/\sqrt{t_n}$ for $n\in\N$ and note that
\[
\zeta_T
=\left(\frac{\sum_{n\in\GI_T} (\eta_n^2-\sigma_{t_n}^2)}{2\sqrt{\log T}},
\frac{|\GI_T|-\log T}{\sqrt{\log T}},
\sum_{n=1}^\infty \sqrt{\ell_n}\eta_n^+,
\sum_{n=1}^\infty \sqrt{\ell_n}\eta_n,
\sum_{n=1}^\infty \ell_n\1_{\{\eta_n>0\}}
\right).
\]
Let $W_1,W_2\ldots$ be an iid sequence of standard normal random variables independent of $\ell$, $Z_1$ and $Z_2$. For any $k\in\N$ and $T>1$ define the random variables
\begin{gather*}
\chi_{k,T}
\coloneqq\left(\frac{\sum_{n=k}^\infty (\eta_n^2-\sigma_{t_n}^2)\1_{\{t_n\ge 1\}}}{2\sqrt{\log T}},
\frac{\sum_{n=k}^\infty\1_{\{t_n\ge 1\}}-\log T}{\sqrt{\log T}},
\sum_{n=1}^{k-1} \sqrt{\ell_n}\eta_n^+,
\sum_{n=1}^{k-1} \sqrt{\ell_n}\eta_n,
\sum_{n=1}^{k-1} \ell_n\1_{\{\eta_n>0\}}
\right),\\
\chi_{k}
\coloneqq\left(\frac{\sigma^2}{\sqrt{2}}Z_1,
Z_2,
\sum_{n=1}^{k-1} \sqrt{\ell_n}\sigma W_n^+,
\sum_{n=1}^{k-1} \sqrt{\ell_n}\sigma W_n,
\sum_{n=1}^{k-1} \ell_n\1_{\{\sigma W_n>0\}}
\right).
\end{gather*}
By~\cite[Thm~3.2]{BillingsleyConvergence},~\eqref{eq:quintuple-conv} will follow if we prove that the following limits hold:
\begin{equation}
\label{eq:quintuple-Billingsley}
\text{(a)}~\chi_{k,T}\xrightarrow[T\to\infty]{d}\chi_k,
\quad
\text{(b)}~\chi_k\xrightarrow[k\to\infty]{d}\zeta,
\quad
\text{(c)}~\lim_{k\to\infty}\limsup_{T\to\infty}
\p(\|\chi_{k,T}-\zeta_T\|>\epsilon)=0,
\enskip\text{for all }\epsilon>0,
\end{equation}
where $\|x\|=\sum_{i=1}^d |x_i|$ denotes the $\ell^1$-norm in $\R^d$, $d\ge 1$. 

\textbf{Step 2.} In this step we establish~\hyperref[eq:quintuple-Billingsley]{(23a)}. Define $\ell^{(k)}\coloneqq(\ell_1,\ldots,\ell_{k-1})$. To prove~\hyperref[eq:quintuple-Billingsley]{(23a)}, it suffices to show that $\E[\phi(\chi_{k,T})|\ell^{(k)}]\to \E[\phi(\chi_k)|\ell^{(k)}]$ a.s. as $T\to\infty$ for any continuous and bounded function $\phi:\R^5\to\R$. With this in mind, denote by $\p^{(k)}$ the conditional probability measure $\p$ given $\ell^{(k)}$.

Under $\p^{(k)}$, the process $(\ell_k,\ell_{k+1},\ldots)$ is a uniform stick-breaking process on $[0,L_{k-1}]$ independent of the variables $(\eta_n)_{n<k}$. Thus the first two coordinates of $\chi_{k,T}$ independent under $\p^{(k)}$ of the last three coordinates. Moreover, since $X_t/\sqrt{t}\cid \sigma Z_1$ as $t\to\infty$, then, under $\p^{(k)}$, we have $(\eta_1,\ldots,\eta_{k-1}) =(\xi_1/\sqrt{t_1},\ldots,\xi_{k-1}/\sqrt{t_{k-1}}) \cid(\sigma W_1,\ldots,\sigma W_{k-1})$ as $T\to\infty$ (recall that $t_n=T\ell_n$). Thus, to prove~\hyperref[eq:quintuple-Billingsley]{(23a)}, it suffices to show that the first two coordinates of $\chi_{k,T}$ converge weakly to the first two coordinates of $\chi_k$ under $\p^{(k)}$. 

Recall that, under $\p^{(k)}$, the process $(\ell_k,\ell_{k+1},\ldots)$ is a uniform stick-breaking process on $[0,L_{k-1}]$ and $\sum_{n=k}^\infty t_n=TL_{k-1}$. Thus, Proposition~\ref{prop:sizeof_n_T} implies that 
\[
\frac{\sum_{n=k}^\infty\1_{\{t_n\ge 1\}}-\log (TL_{k-1})}{\sqrt{\log (TL_{k-1})}}
\cid Z_2,
\quad \text{as $T\to\infty$ under $\p^{(k)}$.}
\]
Since $\log(TL_{k-1})=\log T+\log L_{k-1}$, where $L_{k-1}$ is deterministic under $\p^{(k)}$, then
\[
M_T\coloneqq\frac{\sum_{n=k}^\infty\1_{\{t_n\ge 1\}}-\log T}{\sqrt{\log T}}
\cid Z_2,
\quad \text{as $T\to\infty$ under $\p^{(k)}$.}
\]
Moreover, since $\p^{(k)}(\cdot|\ell)=\p(\cdot|\ell)$, Proposition~\ref{prop:conditional_CLT} implies that $\p^{(k)}(\Sigma_T\le x|\ell)\ciL\p(\sigma^2Z_1/\sqrt{2}\le x)$ for all $x\in\R$ as $T\to\infty$, where $\Sigma_T$ is as in~\eqref{eq:Upsilon-approx}. Denote by $\E^{(k)}$ the expectation under $\p^{(k)}$. Thus, taking limits in the following identity 
\begin{align*}
\E^{(k)} [\1_{\{M_T\le y\}}\p^{(k)}(\Sigma_T\le x|\ell)]
&=
\p^{(k)}(M_T\le y)\p^{(k)}(\sigma^2Z_1/2\le x)\\
&\qquad+\E^{(k)}[\1_{\{M_T\le y\}}
(\p^{(k)}(\Sigma_T\le x|\ell)-\p^{(k)}(\sigma^2Z_1/\sqrt{2}\le x))],
\end{align*}
implies that $\p^{(k)}(M_T\le y,\Sigma_T\le x)\to \p^{(k)}(Z_2\le y)\p^{(k)}(\sigma^2Z_1/\sqrt{2}\le x)$ as $T\to\infty$.
To see that the first two coordinates of $\chi_{k,T}$ converge weakly to the first two coordinates of $\chi_k$ under $\p^{(k)}$, it suffices to note that $\E^{(k)}\sum_{n=1}^{k-1} |\eta_n^2-\sigma_{t_n}^2|/\sqrt{\log T}
\le 2(k-1)\sigma^2/\sqrt{\log T}\to 0$ as $T\to\infty$. 

\textbf{Step 3.}
In this step we establish~\hyperref[eq:quintuple-Billingsley]{(23b)}--\hyperref[eq:quintuple-Billingsley]{(23c)}. To prove~\hyperref[eq:quintuple-Billingsley]{(23b)}, it suffices to show the convergence for the last three coordinates. Note that 
\[
\sum_{n=1}^{k-1}\big(\sqrt{\ell_n}\sigma W_n^+, \sqrt{\ell_n}\sigma W_n, \ell_n\1_{\{\sigma W_n>0\}}\big)
\xrightarrow[k\to\infty]{a.s.}
\sum_{n=1}^\infty(\sqrt{\ell_n}\sigma W_n^+, \sqrt{\ell_n}\sigma W_n, \ell_n\1_{\{\sigma W_n>0\}}),
\]
where the limit has the same law as $(\sigma \ov B_1,\sigma B_1,\rho)$ by the scaling property of Brownian motion and~\eqref{eq:SB-representation} applied to $\sigma B$, implying~\hyperref[eq:quintuple-Billingsley]{(23b)}.

By Markov's inequality,  \hyperref[eq:quintuple-Billingsley]{(23c)} will follow if we prove that $\lim_{m\to\infty}\limsup_{T\to\infty} \E\|\chi_{k,T}-\zeta_T\|=0$. Moreover, the previous limit is a simple consequence of the following limits
\begin{align*}
\limsup_{T\to\infty}\frac{\E\sum_{n=1}^{k-1} |\eta_n^2-\sigma_{t_n}^2|\1_{\{t_n\ge 1\}}}{2\sqrt{\log T}}
=0,&&
\limsup_{T\to\infty}\frac{\E\sum_{n=1}^{k-1} \1_{\{t_n\ge 1\}}}{2\sqrt{\log T}}
=0,&\\
\lim_{k\to\infty}\limsup_{T\to\infty}\E\sum_{n=k}^\infty\sqrt{\ell_n}|\eta_n|
=0,&&
\lim_{k\to\infty}\E\sum_{n=k}^\infty\ell_n
=0.&
\end{align*}
The first two limits in the display are obvious. The fourth limit holds since $\sum_{n=k}^\infty\ell_k=L_{k-1}$ and $\E L_{k-1}=2^{1-k}$. Finally, the third limit in the display above follows from the bounds
\[
\E\sum_{n=k}^\infty\sqrt{\ell_n}|\eta_n|
\le\sum_{n=k}^\infty\E[\sqrt{\ell_n}\E_\ell\big[\eta_n^2]^{1/2}\big]
=\sigma\sum_{n=k}^\infty\E\sqrt{\ell_n}
=\sigma\sum_{n=k}^\infty(2/3)^n
=3\sigma (2/3)^{k},
\]
implying~\hyperref[eq:quintuple-Billingsley]{(23c)} and completing the proof.

\end{proof}

\begin{proof}[Proof of Corollary~\ref{cor:2nd-moment}]
By Theorem~\ref{thm:maintheorem1}, it suffices to prove the claims on $\int_\R x^2\log^+(\min\{T, x^2\})\nu(dx)$. Since $x^2\log^+(\min\{T, x^2\})/\log T$ tends to $0$ pointwise on $x$ as $T\to\infty$ and is upper bounded by the $\nu$-integrable function $x\mapsto x^2$, the dominated convergence theorem implies that the integral is $\oh(\log T)$. Similarly, the integral is $\oh(\sqrt{\log T})$ if $x\mapsto x^2(\log^+|x|)^{1/2}$ is $\nu$-integrable.
\end{proof}

\section{Stable domain of attraction}
\label{sec:stable_domain}
This section is dedicated to proving Theorems~\ref{thm:Theorem2}, \ref{thm:Theorem4} and~\ref{thm:Theorem3}, stated in Section~\ref{sec:intro}. Assume that the limit in~\eqref{eq:attract} holds for some $\alpha\in(0,2]\setminus\{1\}$. Recall that this is equivalent to
\begin{equation}
\label{eq:attraction}
(X_{tT}/a_T)_{t\in[0,1]} 
    \cid (S_\alpha(t))_{t\in[0,1]},
\qquad\text{as }T\to\infty,
\end{equation} 
in the Skorokhod space $\mathcal{D}[0,1]$ equipped with the $J_1$-topology~\cite[Ch.~3]{BillingsleyConvergence}, where $a_T$ is as in~\eqref{eq:attract}. Since $a_T\to\infty$ as $T\to\infty$, we assume without loss of generality that $a_T>1$ is locally bounded for all $T\ge 1$. The following lemma provides a key step in the proofs of Theorems~\ref{thm:Theorem2} and~\ref{thm:Theorem4}. 

\begin{lemma}
\label{lem:stable_mom}
Suppose a L\'evy process $X$ satisfies~\eqref{eq:attraction} for some $\alpha\in(0,2]$. Then, for every $p\in[0,\alpha)$, there exists a constant $C_p\in(0,\infty)$ such that $\E[|X_t/a_t|^p]\le C_p$ for all $t\ge 1$. 
\end{lemma}

\begin{proof}
By the the concavity of $x\mapsto x^p$ (when $p\in[0, 1]$) and Jensen's inequality (when $p\in(1,\alpha)$), we have $(a+b)^p\le 2^{(p-1)^+}(a^p+b^p)$ for any $a,b\ge 0$. Thus, $\E[|X_t|^p]\le 2^{(p-1)^+}(\E[|X_{\floor{t}}|^p]+\E[|X_{t-\floor{t}}|^p])$ for all $t\ge 1$, where $\floor{t}\coloneqq \sup\{m\in\N:m\le t\}$. By~\cite[Lem.~5.2.2]{IbragimovLinnikBook}, $\E[|X_n/a_n|^p]$ is bounded for all $n\in\N$. By the regular variation of $a_t\ge 1$, we have 
\[
1\le\liminf_{t\to\infty}\frac{a_t}{a_{\floor{t}}}
\le\limsup_{t\to\infty}\frac{a_t}{a_{\floor{t}}}
\le\limsup_{t\to\infty}\frac{a_t}{a_{ct}}=c^{-1/\alpha},
\]
for any $c\in(0,1)$, implying $a_t/a_{\floor{t}}\to 1$ as $t\to\infty$. Thus, it suffices to show that $\E[|X_s|^p]$ is bounded for $s\in[0,1]$. This bound follows directly from~\cite[Lem.~2]{LevySupSim} and the inequality $\E[|X_s|^p]\le\E[\ov{X}_s^p] +\E[|\un{X}_s|^p]$ implied by $|X_s|^p\leq \max\{\ov{X}_s^p,|\un{X}_s|^p\}$. 
\end{proof}

\begin{remark}\nf
An explicit upper bound in Lemma~\ref{lem:stable_mom} can be obtained in terms of the characteristics of $X$ and the regularly varying function $a_t$ by using methods analogous to the ones in the proof of~\cite[Lem.~2]{LevySupSim}. Since the explicit value of the upper bound $C_p$ is not important in our context, we only provide the short proof above. 
\end{remark}

\subsection{The case of finite mean}

\begin{proof}[Proof of Theorem~\ref{thm:Theorem2}]
Recall $\p_\ell(\cdot)=\p(\cdot|\ell)$ and $\E_\ell[\cdot]=\E[\cdot|\ell]$, where $\ell$ is the stick-breaking process on $[0,1]$, and $t_n=T\ell_n$. 
Denote $\eta_n\coloneqq\xi_n/a_{t_n}$ and $\varrho_n\coloneqq a_{t_n}/a_T$ for $n\in\N$ and note that $\sqrt{t_n^2+\xi_{n}^2}-t_n =\xi_{n}^2/(t_n+\sqrt{t_n^2+\xi_{n}^2})$. Thus, by~\eqref{eq:SB-representation}, we have
\[
\left(
\frac{\Upsilon_T^\frown-T}{a_T^2/T},
\frac{\ov C_T^\frown}{a_T},
\frac{C_T^\frown(T)}{a_T},
\frac{\gamma_T^\frown}{T} 
\right)
\eqd
\sum_{n=1}^\infty\bigg(
\frac{\varrho_n^2\eta_n^2}{\ell_n+\sqrt{\ell_n^2+\varrho_n^2\eta_{n}^2a_T^2/T^2}},\varrho_n\eta_{n}^+,\varrho_n\eta_{n},
\ell_n\1_{\{\varrho_n\eta_{n}>0\}} 
\bigg).
\]
By~\cite[Thm~3.2]{BillingsleyConvergence}, \eqref{eq:Theorem1.2convergence} will follow if we prove the following limits: for any $k\in\N$ we have 
\begin{equation}
\label{reducedconvergencestable}
\begin{split}
&\sum_{n=1}^{k-1}\bigg(
\frac{\varrho_n^2\eta_n^2}{\ell_n+\sqrt{\ell_n^2+\varrho_n^2\eta_{n}^2a_T^2/T^2}},\varrho_n\eta_{n}^+,\varrho_n\eta_{n},
\ell_n\1_{\{\varrho_n\eta_{n}>0\}} 
\bigg) \\  
&\qquad\xrightarrow[T\to\infty]{d}
\sum_{n=1}^{k-1} \bigg(
\frac{1}{2}\ell_n^{2/\alpha-1}\big(S_\alpha^{(n)}\big)^2,
\ell_n^{1/\alpha} \big( S_\alpha^{(n)}\big)^+,
\ell_n^{1/\alpha} S_\alpha^{(n)},
\ell_n\1_{\{S_\alpha^{(n)}>0\}}
\bigg),
\end{split}
\end{equation} 
and, for all $\epsilon>0$, 
\begin{equation}
\label{limreducedlimitstable}
\begin{gathered}
\lim_{k\to\infty}\limsup_{T\to\infty}
    \p\bigg(\sum_{n=k}^\infty\big\|\big(
R_n,\varrho_n\eta_{n}^+,\varrho_n\eta_{n},
\ell_n\1_{\{\varrho_n\eta_{n}>0\}} 
\big)\big\|>\epsilon\bigg)
=0,\\
\text{where}\quad
R_n\coloneqq\frac{\varrho_n^2\eta_n^2}{\ell_n+\sqrt{\ell_n^2+\varrho_n^2\eta_n^2a_T^2/T^2}},
\end{gathered}
\end{equation}
and $\|x\|=\sum_{i=1}^d |x_i|$ denotes the $\ell^1$-norm in $\R^d$, $d\ge 1$. 

To prove~\eqref{reducedconvergencestable}, it suffices to show that the weak convergence holds conditional on $\ell$. By assumption, we have $X_t/a_t\cid S_\alpha^{(1)}$, $a_{ct}/a_t\to c^{1/\alpha}$ and $a_t/t\to 0$ as $t\to\infty$. Thus, given $\ell$, the random variables $\eta_1,\ldots,\eta_k$ are independent and we have the following convergences as $T\to\infty$: 
$(\eta_1,\ldots,\eta_k) \cid (S_\alpha^{(1)},\ldots,S_\alpha^{(k)})$, $(\varrho_1,\ldots,\varrho_k) \to(\ell_1^{1/\alpha},\ldots,\ell_k^{1/\alpha})$ and $a_T/T\to 0$. The continuous mapping theorem then yields the weak convergence in~\eqref{reducedconvergencestable} conditional on $\ell$. 

Next we prove~\eqref{limreducedlimitstable}. Note that $\sum_{n=k}^\infty \ell_k=L_{k-1}$ and $\p(L_{k-1}>\epsilon)\le \epsilon^{-1}\E L_{k-1}\to 0$ as $k\to\infty$, so it suffices to show that, for all $\epsilon>0$, the following limits hold as $k\to\infty$:
\begin{equation}
\label{eq:tightness_alpha}
\limsup_{T\to\infty}
\p\left(\sum_{n=k}^\infty R_n >\epsilon\right)
\to0,
\quad
\limsup_{T\to\infty}
\p\left(\sum_{n=k}^\infty\varrho_n|\eta_n|    >\epsilon\right)
\to0.
\end{equation} 
We will prove both limits via Markov's inequality $\p(|\zeta|>\epsilon)\le \epsilon^{-p}\E[|\zeta|^p]$ for $p>0$, and bounding the first moment by splitting the summation over the sets $\GI_T$ and $\GI_T^\co$ (recall that $\GI_T=\{n\in\N:t_n\ge 1\}$). First note that $R_n\le |\xi_n|(T/a_T^2)$ and $\rho_n|\eta_n|=|\xi_n|/a_T$, where $a_T\to\infty$ and $a_T^2/T\to\infty$ as $T\to\infty$. There exists a constant $K$ such $\E[|X_t|]\le K\sqrt{t}$ for all $t\le1$ (see, e.g.~\cite[Lem.~2]{LevySupSim}), so Corollary~\ref{cor:int_SB_tauT}(a) yields
\begin{align*}
&\limsup_{T\to\infty}
\E\sum_{n\in\GI_T^\co, n\ge k}
    \varrho_n\E_\ell|\eta_n|
\le \limsup_{T\to\infty}
\frac{K}{a_T}\E\sum_{n\in\GI_T^\co} \ell_n^{1/2}
=\limsup_{T\to\infty}\frac{2K}{a_T}=0,
\quad\text{and}\\
&\limsup_{T\to\infty}
\E\sum_{n\in\GI_T^\co, n\ge k}R_n
\le \limsup_{T\to\infty}
\frac{KT}{a_T^2}\E\sum_{n\in\GI_T^\co} \ell_n^{1/2}
=\limsup_{T\to\infty}\frac{2KT}{a_T^2}=0.
\end{align*}

It remains to consider the summation sets $\GI_T\cap\{k,k+1,\ldots\}$. 
By Lemma~\ref{lem:stable_mom}, for any $p\in(0,\alpha)$, we have $\E_\ell[|\eta_n|^p]\le C_p$ for some $C_p>0$.
Since $t\mapsto a_t$ is regularly varying at infinity with index $1/\alpha$, Potter's theorem~\cite[Thm~1.5.6]{BinghamBook} implies that for all $q\in(0,1/\alpha)$ there exists a constant $C'_{q}>0$ such that $a_s/a_t\le C'_{q}(s/t)^{q}$ for all $t>s\ge 1$. Thus, the second limit in~\eqref{eq:tightness_alpha} follows from the limit
\[
\limsup_{T\to\infty}
\E\sum_{n\in\GI_T, n\ge k}
    \varrho_n\E_\ell|\eta_n|
\le C_1C'_{1/2}
\sum_{n=k}^\infty \E[\ell_n^{1/2}]
=3C_1C'_{1/2}(2/3)^{k-1}
\xrightarrow[k\to\infty]{}0.
\]

Next note that $R_n\le\varrho_n^2\eta_n^2/\ell_n$ and fix any $p\in(0,\alpha/2)$ and $q\in(1/2,1/\alpha)$. By Markov's inequality and the subadditivity of $x\mapsto x^p$, the first limit in~\eqref{eq:tightness_alpha} follows from
\[
\limsup_{T\to\infty}\E\sum_{n\in\GI_T,n\ge k}R_n^p
\le C_{2p}(C'_{q})^{2p}\sum_{n=k}^\infty\E[\ell_n^{p(2q-1)}]
=\frac{C_{2p}(C'_{q})^{2p}(1+p(2q-1))^{1-k}}{p(2q-1)}
\xrightarrow[k\to\infty]{}0.
\qedhere
\]
\end{proof}

The asymptotic equivalence $f(x)\sim g(x)$ as $x\to\infty$ is defined as $\lim_{x\to\infty}f(x)/g(x)=1$. 

\begin{proof}[Proof of Proposition \ref{prop:tailbehaviour}]
Note that the random variable $Q\coloneqq \frac{1}{2}\sum_{n=1}^\infty \ell_n^{2/\alpha-1}(S_\alpha^{(n)})^2$ satisfies
\[
2Q
=\ell_1^{2/\alpha-1}(S_\alpha^{(1)})^2+\sum_{i=2}^\infty \ell_n^{2/\alpha-1}(S_\alpha^{(n)})^2 
= \ell_1^{2/\alpha-1}(S_\alpha^{(1)})^2+L_1^{2/\alpha-1}\sum_{i=2}^\infty \left(\frac{\ell_n}{L_1}\right)^{2/\alpha-1}(S_\alpha^{(n)})^2.
\] 
Let $A\coloneqq L_1^{2/\alpha-1}$, 
$B\coloneqq\frac{1}{2}\ell_1^{2/\alpha-1}(S_\alpha^{(1)})^2$ and $Q'\coloneqq\frac{1}{2}\sum_{i=2}^\infty 
(\ell_n/L_1)^{2/\alpha-1}(S_\alpha^{(n)})^2$ and note that $Q=AQ'+B$. Since $(\ell_n/L_1)_{n\ge 2}$ is a stick-breaking process on $[0,1]$ independent of $L_1$ and $S_\alpha^{(1)}$, we conclude that $Q'\eqd Q$ is independent of $(A,B)$. 

By~\cite[Thm~2.4.3]{PerpetuityBook} it follows that $\p(Q>x)\sim(1-\E[A^{\alpha/2}])^{-1}\p(B>x)$, as $x\to\infty$. Furthermore, by~\cite[Lem.~B.5.1]{PerpetuityBook}, we have 
\begin{align*}
 \p(B>x)\sim 
 \E\big[\big(\tfrac{1}{2}\ell_1^{2/\alpha-1} 
    \big)^{\alpha/2}\big] 
 \p((S_\alpha^{(1)})^2>x), 
 \quad\text{ as }
 x \to \infty.
\end{align*} 
Recall that $L_1 =1-\ell_1\sim U(0,1)$. Similarly, we have that $\ell_1\sim U(0,1)$. Thus, it follows that
\begin{align*}
\big( 1-\E [A^{\alpha/2}]\big)^{-1}
    \E\big[\big(\tfrac{1}{2}\ell_1^{2/\alpha-1} 
        \big)^{\alpha/2}\big]
&=2^{-\alpha/2}\big(1-\E\big[V_1^{1-\alpha/2}
    \big]\big)^{-1}
\E\big[V_1^{1-\alpha/2}\big]\\
&=2^{-\alpha/2}\Big( 1-\frac{2}{4-\alpha}\Big)^{-1}
\frac{2}{4-\alpha}
=\frac{2^{1-\alpha/2}}{2-\alpha}.
\end{align*} 
Thus we have $\p(Q>x)\sim 2^{1-\alpha/2} \p ((S_\alpha^{(1)})^2>x)/(2-\alpha)$, as $x\to\infty$. The last asymptotic equivalence in Proposition~\ref{prop:tailbehaviour} follows from the identity $\p((S_\alpha^{(1)})^2>x)
=\p (S_\alpha^{(1)}>\sqrt{x})+\p (-S_\alpha^{(1)}>\sqrt{x})$.
\end{proof}

\begin{proof}[Proof of Theorem~\ref{thm:Theorem4}]
(a) Assume $\mu >0$. We assume without loss of generality that $t\mapsto a_t$ is continuous and $a_t\ge 1$ for all $t>0$. Define 
\[
Z_T\coloneqq \bigg(\frac{\mu}{\sqrt{1+\mu^2}},1,1\bigg)
\frac{X_T-\mu T}{a_T}
\qquad\text{and}\qquad 
Z'_T\coloneqq \frac{1}{a_T}
    \left(\Upsilon_T^\frown -\sqrt{1+\mu^2}T,X_T-\mu T,\ov{X}_T-\mu T\right).
\]
Since $Z_T\cid(\mu/\sqrt{1+\mu^2},1,1)S_\alpha(1)$ as $T\to\infty$, it suffices to show that $\|Z_T-Z'_T\|\cip 0$ as $T \to \infty$. Define 
\[
\Delta_T
\coloneqq \Upsilon_T^\frown 
    -\sqrt{1+\mu^2}T
    -\frac{\mu}{\sqrt{1+\mu^2}}(X_T-\mu T),
\qquad T>0.
\] 
Note that $|\un{X}_T/a_T|\cip 0$ as $T\to\infty$ since the positive drift $\mu >0$ implies that $-\un{X}_T\to-\un{X}_\infty<\infty$ a.s. as $T\to\infty$. Since $\|Z'_T-Z_T\|
=a_T^{-1}\|(\Delta_T,0,(\ov{X}_T-X_T)/a_T)\|$, and $\ov{X}_T-X_T\eqd-\un{X}_T$, part~(a) will follow if we show that $\Delta_T/a_T\cip 0$ as $T\to\infty$.

By~\eqref{eq:SB-representation}, we have $(\Upsilon_T^\frown -T,X_T-\mu T)\eqd\sum_{n=1}^\infty( \sqrt{t_n^2+\xi_n^2}-t_n,\wt\xi_n)$, where  $\wt\xi_n\coloneqq\xi_n-\mu t_n$. Thus we have  $\Delta_T\eqd\sum_{n\in\N}\zeta_n$, where 
\[
\zeta_n
\coloneqq \sqrt{t_n^2+\xi_n^2}
-\sqrt{1+\mu^2}t_n-\frac{\mu}{\sqrt{1+\mu^2}}\wt\xi_{t_n} =\sqrt{1+\mu^2}t_n 
\bigg( \bigg(1+\frac{\wt\xi_n^2 
+ 2\mu t_n \wt\xi_n}{t_n^2(1+\mu^2)}\bigg)^{1/2}-1
-\frac{\mu}{1+\mu^2}\frac{\wt\xi_n}{t_n}\bigg).
\]
To prove that $\Delta_T/a_T\cip 0$, we again split the summation set with $\GI_T$ and $\GI_T^\co$. Define: 
\[
\Delta_T^{(1)} 
\coloneqq \sum_{n\in\GI_T} \zeta_n,
\qquad
\Delta_T^{(2)} 
\coloneqq \sum_{n\in\GI_T^\co} \zeta_n,
\]
and note that $\Delta_T\eqd\Delta_T^{(1)}+\Delta_T^{(2)}$.

Fix some $p\in(0,\alpha/2)$ and use the inequality $\sqrt{1+z}\le 1+z/2$ for $z\ge -1$ and the subadditivity of $x\mapsto x^p$ to obtain 
\[
\E[|\Delta_T^{(1)}/a_T|^p]
\le\E\bigg[\bigg|\sum_{n\in\GI_T}
    \frac{\wt\xi_n^2}{2a_T\sqrt{1+\mu^2}t_n}\bigg|^p
        \bigg]
\le\E\sum_{n\in\GI_T}
        \frac{|\wt\xi_n|^{2p}}{a_T^p t_n^p}.
\]
Recall that $(X_t-\mu t)/a_t\cid S_\alpha(1)$ as $t\to\infty$. Thus, by Lemma~\ref{lem:stable_mom}, there exists a constant $C_{2p}>0$ such that $\E[|X_t-\mu t|^{2p}]\leq C_{2p} a_t^{2p}$ for all $t\ge 1$. Therefore $\E_\ell[|\wt{\xi}_{n}|^{2p}]\le C_{2p} a_{t_n}^{2p}$ for $n\in\GI_T$. 

Suppose $\alpha\in(1,2)$. Pick $q\in(1/2,1/\alpha)$ and apply Potter's Theorem~\cite[Thm~1.5.6]{BinghamBook} to obtain $a_t/a_T\le C'_q (t/T)^q$ for all $T>t\ge 1$ and some $C'_q>0$. Thus, Corollary~\ref{cor:int_SB_tauT}(a) yields
\begin{align*}
\E[|\Delta_T^{(1)}/a_T|^p]
\le C_{2p}\E\sum_{n\in\GI_T}\frac{a_{t_n}^{2p}}{a_T^pt_n^p}
&=C_{2p}\Big(\frac{a_T}{T}\Big)^p
    \E\sum_{n\in\GI_T}
        \ell_n^{-p}\Big(\frac{a_{t_n}}{a_T}\Big)^{2p}\\
&\le C_{2p}(C'_q)^{2p}\Big(\frac{a_T}{T}\Big)^p
    \E\sum_{n=1}^\infty\ell_n^{p(2q-1)}
=\frac{C_{2p}(C'_q)^{2p}}{p(2q-1)}\Big(\frac{a_T}{T}\Big)^p,
\end{align*}
which tends to $0$ as $T\to\infty$, implying $\Delta_T^{(1)}/a_T\cip 0$. 

Suppose $\alpha=2$. We may assume $a_t=\sqrt{t}l(t)$ for a locally bounded and slowly varying function $l$. Thus, by~\cite[Prop.~1.5.9a]{BinghamBook}, $\tilde l(T)\coloneqq\int_1^T t^{-1}l(t)^{2p}dt$ is also slowly varying and Corollary~\ref{cor:int_SB_tauT}(a) yields
\[
\E[|\Delta_T^{(1)}/a_T|^p]
\le C_{2p}\E\sum_{n\in\GI_T}
    \frac{a_{t_n}^{2p}}{a_T^pt_n^p}
= C_{2p}\E\sum_{n\in\GI_T}
    \frac{l(t_n)^{2p}}{a_T^p}
=C_{2p}\frac{\tilde l(T)}{a_T^p}
\xrightarrow[T\to\infty]{} 0.
\]

It remains to show that $\Delta_T^{(2)}/a_T\cip 0$ 
as $T\to\infty$. The inequality $\sqrt{1+x+y}\geq 1+y/2$ for $x \geq y^2/4$ and $x+y\geq -1$ shows that $\Delta_T^{(2)}\ge 0$ a.s. By the subadditivity of $x\mapsto\sqrt{x}$, we obtain 
\begin{align*}
\frac{1}{a_T}\E[|\Delta_T^{(2)}|]
&\le\frac{\sqrt{1+\mu^2}}{a_T}\E\sum_{n\in \GI_T^\co}t_n 
\bigg| \frac{\wt\xi_n}{t_n\sqrt{1+\mu^2}}
    +\frac{(2|\mu||\wt\xi_n|)^{1/2}}{\sqrt{t_n(1+\mu^2)}}
    -\frac{\mu}{1+\mu^2}\frac{\wt\xi_n}{t_n}\bigg|\\
&\leq \frac{1}{a_T}\E\sum_{n\in \GI_T^\co}\bigg( 
    \bigg( 1-\frac{\mu}{\sqrt{1+\mu^2}}\bigg)|\wt\xi_n|
    +\sqrt{2|\mu|t_n|\wt\xi_n|} \bigg).
\end{align*} 
By~\cite[Eq.~(24)]{LevySupSim} and Jensen's inequality, there exists a constant $C>0$ such that $\E[|X_t-\mu t|]\le C\sqrt{t}$ for all $t\le 1$. Thus,  Corollary~\ref{cor:int_SB_tauT}(a) yields $\Delta_T^{(2)}/a_T\ciL 0$ as $T\to\infty$, completing the proof of part~(a).

(b) Note that $\ov{X}_T\to\ov{X}_\infty <\infty$ a.s. and $\gamma_T^\frown\to\gamma_\infty^\frown<\infty$ a.s. as $T\to\infty$. We next split the length of the concave majorant in two at the time of the supremum, so the total length $\Upsilon_T^\frown$ up to time $T$ is equal to the sum of the length $\Delta_T^{(1)}$ up to time $\gamma_T^\frown$ and the length $\Delta_T^{(2)}$ from $\gamma_T^\frown$ to $T$. It follows that $\Delta_T^{(1)}\to\Upsilon_{\gamma_\infty^\frown}^\frown$ a.s. as $T\to\infty$, implying $\Delta_T^{(1)}/a_T\to 0$ a.s. Thus, it suffices to consider $\Delta_T^{(2)}$ for the weak limit of $\Upsilon_T^\frown$. Since the post-supremum process is independent of the pre-supremum process by~\cite[Thm~2.3]{BertoinSplittingSup}, we conclude, as in part~(a), that 
\[
\bigg(\frac{\Delta_T^{(2)}
    -(T-\gamma_T^\frown)}{a_T},
\frac{(C_T^\frown(T)-\ov{X}_T)
    -\mu(T-\gamma_\infty^\frown)}{a_T}\bigg)\bigg| 
        (\ov{X}_\infty,\gamma_\infty^\frown) 
\cid\bigg(\frac{\mu}{\sqrt{1+\mu^2}},1\bigg)S_\alpha(1), 
\quad\text{ as }T\to\infty.
\]
Note here that the limit law does not depend on 
$(\ov{X}_\infty,\gamma_\infty^\frown)$, so the limit is 
independent of $(\ov{X}_\infty,\gamma_\infty^\frown)$. 
Since we also have $|\ov X_\infty-\ov X_T|\to 0$ 
and $|\gamma_\infty^\frown-\gamma_T^\frown|\to 0$ a.s. 
as $T\to\infty$, the result follows. 
\end{proof}

\subsection{Sandwiching the concave majorant}
\label{subsec:proof_alpha_less_than_1}
When the tails of  $X$ are sufficiently heavy for it not to have the first moment, the asymptotic behaviour of the boundary of its convex hull is straightforward.

\begin{proof}[Proof of Theorem~\ref{thm:Theorem3}]
The supremum, infimum and the times at which they are attained are functionals that are continuous a.s. in  $J_1$-topology with respect to the law of an  $\alpha$-stable process, since the times at which the extrema are attained are a.s. unique (see~\cite[Lem.~14.12]{MR1876169}  and~\cite[Thm~2]{MR2978134}). Thus, 
by the continuous mapping theorem, it suffices to prove
 $|\Upsilon_T^\frown-(2\ov{C}_T^\frown-C_T^\frown(T))|/a_T\to 0$ and $|\Upsilon_T^\smile-(C_T^\smile(T)-2\un{C}_T^\smile)|/a_T\to 0$ a.s. as $T\to\infty$. 
 Recall 
 $X_T=C_T^\frown(T)\le \ov X_T = \ov{C}_T^\frown$
 and
 $\gamma^\frown_T\in[0,T]$.
 Hence, 
by Figure~\ref{fig:FacesCM},
the following inequalities hold:
\begin{align*}
2\ov X_T - X_T
\le
((\gamma^\frown_T)^2 + (\ov{X}_T)^2)^{1/2}+
((T-\gamma^\frown_T)^2 + (\ov{X}_T-X_T)^2)^{1/2}
 \leq \Upsilon^\frown_T 
\leq 
2\ov X_T - X_T +T.
\end{align*} 
Since $\alpha\in(0,1)$ we have $\lim_{T\to\infty} T/a_T= 0$,
implying 
$|\Upsilon_T^\frown-(2\ov{C}_T^\frown-C_T^\frown(T))|/a_T\to 0$
a.s. as $T\to\infty$. The proof of the second limit is analogous.
\end{proof}

\begin{proof}[Proof of 
Proposition~\ref{prop:compare_length}]
(a)\&(b) In part (a), define $a_T\coloneqq\sqrt{T}$ for all $T>0$. Note that $\Upsilon^\sqcap_T-T=2\ov{X}_T-X_T$ and 
\[
\Upsilon^\wedge_T-T
=\Big(\sqrt{(\gamma^\frown_T)^2 
    + \ov{X}_T^2}-\gamma^\frown_T\Big)
+\Big(\sqrt{(T-\gamma^\frown_T)^2 + (\ov{X}_T-X_T)^2}
        -(T-\gamma^\frown_T)\Big).
\]
We will show that 
\begin{equation}
\label{eq:hut-cip}
\frac{T}{a_T^2}\bigg|\Upsilon_T^\wedge-T- \frac{\ov X_T^2}{2\gamma_T^\frown}-\frac{(\ov X_T-X_T)^2}{2(T-\gamma_T^\frown)}\bigg|
\cip 0,\qquad \text{as }T\to\infty.
\end{equation}
The conclusions of parts~(a) \&~(b) will then follow from~\eqref{eq:hut-cip}, an application of the continuous mapping theorem and Theorems~\ref{thm:maintheorem1} \&~\ref{thm:Theorem2}, respectively. 

To prove~\eqref{eq:hut-cip}, by symmetry, it suffices to show that $Ta_T^{-2}|((\gamma^\frown_T)^2 + \ov{X}_T^2)^{1/2}-\gamma^\frown_T - \ov{X}_T^2/(2\gamma^\frown_T)|\cip 0$ as $T\to\infty$. An application of Taylor's theorem yields $\sqrt{1+x^2}=1+x^2/2+x^4\theta(|x|)/8$, where $\theta:[0,\infty)\to[0,1]$ is a bounded function. Thus, the limit in probability is implied by the limit $Ta_T^{-2}\ov{X}_T^4/(\gamma^\frown_T)^3\cip 0$ as $T\to\infty$, which is itself a direct consequence of the fact that $a_T/T\to 0$, the continuous mapping theorem and the weak limits $\gamma^\frown_T/T\cid \gamma^{\alpha\frown}$ and $\ov{X}_T/a_T\cid \ov{S}_\alpha(1)$ as $T \to \infty$. 

(c) The proof follows as in the proof of Theorem~\ref{thm:Theorem3}, using the triangle inequality to obtain 
\[
2\ov{X}_T-X_T
\le\Upsilon^\wedge_T
\le\Upsilon^\sqcap_T
=T+2\ov{X}_T-X_T,
\]
and then using the fact that $T/a_T\to 0$ as $T \to \infty$. 
\end{proof}

\printbibliography

\section*{Acknowledgements}

\thanks{
\noindent AM was supported by EPSRC grant EP/P003818/1 and the Turing Fellowship funded by the Programme on Data-Centric Engineering of Lloyd's Register Foundation;
JGC and AM are supported by The Alan Turing Institute under the EPSRC grant EP/N510129/1; 
JGC is supported by CoNaCyT scholarship 2018-000009-01EXTF-00624 CVU 699336; DB is funded by CDT in Mathematics and Statistics at The University of Warwick.}

\end{document}